\documentclass[12pt]{amsart}
\usepackage{graphicx}
\usepackage{hyperref}
\usepackage{setspace}
\usepackage{enumerate}
\usepackage{amsmath,amsfonts,amssymb,amscd}

\newtheorem{thm}{Theorem}[section]
\newtheorem{cor}[thm]{Corollary}
\newtheorem{lem}{Lemma}
\newtheorem{prop}[thm]{Proposition}

\theoremstyle{definition}

\theoremstyle{remark}

\numberwithin{equation}{section}

\newcommand{\ord}{\textup{ord}\,}
\newcommand{\card}{\textup{card}\,}
\newcommand{\res}{\textup{res}\,}
\newcommand{\bigO}[1]{O \left( #1 \right)}

\begin{document}

\title[The asymptotic formula for Waring's Problem in function fields]{The asymptotic formula for Waring's Problem \\ in function fields}

\author{Shuntaro Yamagishi}
\address{Department of Pure Mathematics \\
University of Waterloo \\
Waterloo, ON\\  N2L 3G1 \\
Canada}
\email{syamagis@uwaterloo.ca}
\indent

\date{Revised on \today}

\begin{abstract}
Let $\mathbb{F}_q[t]$ be the ring of polynomials over $\mathbb{F}_q$, the finite field of $q$ elements, and let $p$ be the characteristic of $\mathbb{F}_q$.
We denote $\widetilde{G}_q(k)$ to be the least integer $t_0$ with the property that for all
$s \geq t_0$, one has the expected asymptotic formula in Waring's problem over $\mathbb{F}_q[t]$
concerning sums of $s$ $k$-th powers of polynomials in $\mathbb{F}_q[t]$.
For each $k$ not divisible by $p$, we derive a minor arc bound from Vinogradov-type estimates, and obtain bounds on $\widetilde{G}_q(k)$ that
are quadratic in $k$, in fact linear in $k$ in some special cases, in contrast to the bounds
that are exponential in $k$ available only when $k < p$. We also obtain estimates related to the slim exceptional sets
associated to the asymptotic formula.
\end{abstract}



\maketitle

\section{Introduction}
\label{Introduction}
In the early twentieth century, Hardy and Littlewood developed the technique now known as the Hardy-Littlewood circle method in a series of papers on Waring's problem. Waring's problem is regarding the representation of a natural number as a sum of integer powers.
More precisely,
given $n, s, k \in \mathbb{N}$, $k \geq 2$, we let
$$
R_{s,k}(n) = \# \{ (x_1, ..., x_s) \in \mathbb{N}^s: x_1^k + ... + x_s^k = n, \ x_i \leq n^{1/k} \ (1 \leq i \leq s) \},
$$
and we consider the smallest number $s$ such that $R_{s,k}(n) > 0$.
There are various questions studied related to Waring's problem, one of which is to find the minimum number of variables required to establish the expected asymptotic formula. This is an important aspect of Waring's problem as a ``brief review of the progress achieved in nearly a
century of development of the Hardy-Littlewood (circle) method reveals that a substantial part has
originated in work devoted to the challenge of establishing the asymptotic formula in Waring's problem'' \cite{W1}.
As stated in \cite{W1}, by a heuristic application of the circle method, one expects that when $k \geq 3$ and $s \geq k+1$,
\begin{equation}
\label{asympt for int}
R_{s,k}(n) = \frac{\Gamma(1 + 1/k)^s}{\Gamma(s/k)}\mathfrak{S}_{s,k}(n) n^{\frac{s}{k}-1} + o(n^{\frac{s}{k}-1}),
\end{equation}
where
$$
\mathfrak{S}_{s,k}(n) = \sum_{a=1}^{\infty} \sum_{ \substack{ a=1 \\ (a,q)=1 }  }^{\infty} \left( \frac{1}{q} \sum_{r=1}^q e^{2 \pi i (a r^k /q) } \right)^s
e^{- 2 \pi i ( n a /q)}.
$$
We note that subject to modest congruence conditions on $n$, one has $1 \ll \mathfrak{S}_{s,k}(n) \ll n^{\varepsilon}$
\cite[Chapter 4]{V1}. Let $\widetilde{G}(k)$ be the least integer $t_0$ with the property that,
for all $s \geq t_0$, and all sufficiently large natural numbers $n$, one has the asymptotic formula ~(\ref{asympt for int}). 
As a consequence of his recent work concerning Vinogradov's mean value theorem, Wooley has significantly improved
estimates on $\widetilde{G}(k)$ \cite{W, W1, W2}.
In particular, it was proved in \cite{W2} that
$ \widetilde{G}(k) \leq 2k^2 - 2k - 8 \ (k \geq 6)$.

In this paper, we consider an analogous problem 
in the setting of $\mathbb{F}_q[t]$,
where $\mathbb{F}_q$ is a finite field of $q$ elements. In other words, we consider the
asymptotic Waring's problem over $\mathbb{F}_q[t]$.
We later define $\widetilde{G}_q(k)$, an analgoue of
$\widetilde{G}(k)$ over $\mathbb{F}_q[t]$, and establish bounds on it. As the function field analogue of
Wooley's work on Vinogradov's mean value theorem \cite{W} has been established in \cite{LW1}
and its multidimensional version in \cite{KLZ}, it is natural to consider its consequences
in improving the number of variables required to establish the asymptotic formula in
Waring's problem over $\mathbb{F}_q[t]$. Here we accomplish this task by taking the approach of \cite{W1}.

Before we can state our main results, we need to introduce notation, some of which we
paraphrase from the material in introduction of \cite{LW}.
We denote the characteristic of $\mathbb{F}_q$, a positive prime number,
by ch$(\mathbb{F}_q) = p$. Unless we specify otherwise,
we always assume $p$ to be the characteristic of $\mathbb{F}_q$ even if it
is not explicitly stated so.
Let $k$ be an integer with $k \geq 2$, let $s \in \mathbb{N}$,
and consider a polynomial $n \in \mathbb{F}_q[t]$. 
We are interested in the representation of $n$ of the form
\begin{equation}
\label{shape}
n = x_1^k + x_2^k + ... + x_s^k,
\end{equation}
where $x_i \in \mathbb{F}_q[t] \ (1 \leq i \leq s)$. It is possible that a representation
of the shape ~(\ref{shape}) is obstructed for every natural number $s$. For example, if the
characteristic $p$ of $\mathbb{F}_q$ divides $k$, then
$ x_1^k + x_2^k + ... + x_s^k =  \left( x_1^{k/p} + x_2^{k/p} + ... + x_s^{k/p} \right)^p$,
and thus $n$ necessarily fails to admit a representation of the shape ~(\ref{shape})
whenever $n \not \in \mathbb{F}_q[t^p]$, no matter how large $s$ may be. In order to
accommodate this and other intrinsic obstructions, we
define $\mathbb{J}_q^k[t]$ to be the additive closure of the set of $k$-th powers of
polynomials in $\mathbb{F}_q[t]$, and we restrict attention to
those $n$ lying in the subring $\mathbb{J}_q^k[t]$ of $\mathbb{F}_q[t]$.
It is also convenient to define $\mathbb{J}_q^k$ to be the additive closure of
the set of $k$-th powers of elements of $\mathbb{F}_q$.

Given $n \in \mathbb{J}_q^k[t]$, we say that $n$ is an \textit{exceptional element} of $\mathbb{J}_q^k[t]$
when its leading coefficient lies in $\mathbb{F}_q \backslash \mathbb{J}_q^k$, and in addition $k$
divides $\deg n$. As explained in \cite{LW}, the strongest constraint on the degrees of the variables that might still permit
the existence of a representation of the shape ~(\ref{shape}) is plainly $\deg x_i \leq \lceil (\deg n) / k \rceil \ (1 \leq i \leq s).$
When $p < k$, however, it is possible that $\mathbb{J}_q^k$
is not equal to $\mathbb{F}_q$, and then the leading coefficient of $n$ need not be an element of $\mathbb{J}_q^k$.
If $k$ divides $\deg n$, so that $n$ is an exceptional polynomial, such circumstances obstruct the existence
of a representation ~(\ref{shape}) of $n$ with variables $x_i$ satisfying the above constraint on their degrees.
For these reasons, following \cite{LW}, we define $P = P_k(n)$ by setting
\begin{equation}
\notag
P =
\begin{cases}
\big{ \lceil } \frac{\deg n}{k}  \big{ \rceil}, & \mbox{if } n \mbox{ is not exceptional}, \\
\frac{\deg n}{k} + 1, & \mbox{if } n \mbox{ is exceptional} .
\end{cases} \\
\end{equation}
In particular, when $n$ is not exceptional, then $P$ is the unique integer satisfying $k(P-1) < \deg n \leq kP$.
We say that $n$ admits a \textit{strict representation} as a sum of $s$ $k$-th powers when for some $x_i \in \mathbb{F}_q[t]$
with $\deg x_i \leq P_k(n) \ (1 \leq i \leq s)$, the equation ~(\ref{shape}) is satisfied.

For notational convenience, let $X = X_{k}(n):= P_k(n) + 1$, and
we define $I_X := \{ x \in \mathbb{F}_q[t] : \deg x < X \}$.
For $n$ a polynomial in $\mathbb{F}_q[t]$,
we denote $R_{s,k}(n)$ to be the number of strict representations of $n$, 
in other words
$$
R_{s,k}(n) = \# \{ (x_1, ..., x_s) \in (I_X)^s : x_1^k + ... + x_s^k = n \}.
$$
Though it is not explicit in the notation, $R_{s,k}(n)$ does depend on $q$.
Suppose the leading coefficient of the polynomial $n$ is $c(n)$. We define
$b = b(n)$ to be $c(n)$ when $k$ divides $\deg n$ and $n$ is not exceptional, and otherwise
we set $b(n)$ to be $0$. In addition, we write $J_{\infty}(n) = J_{\infty}(n;q)$ for the number
of solutions of the equation
$y_1^k + ...+ y_s^k = b$ with $(y_1, ..., y_s) \in \mathbb{F}_q^s \backslash \{ \mathbf{0} \}$.
Analogously to the case of integers, one expects the following asymptotic formula
\begin{equation}
\label{equation baaaaaaahhhhh}
R_{s,k}(n) =  \mathfrak{S}_{s,k}(n) J_{\infty}(n) q^{(s-k)P} + o\left( q^{(s-k)P} \right),
\end{equation}
where
$$
\mathfrak{S}_{s,k}(n) = \sum_{ \substack{   g \in \mathbb{F}_q[t] \\ g \text{ is  monic}  }}
\frac{1}{q^{ s (\deg g) } } \ \sum_{\substack{\deg a < \deg g \\ (a,g)=1 }} \left( \sum_{\deg r < \deg g} e(ar^k/g)\right)^s e(-na/g),
$$
to hold whenever $s$ is sufficiently large with respect to $k$. 
We postpone the definition
of the exponential function $e(\cdot)$ to Section \ref{Prelim}.
By making the circle method applicable over $\mathbb{F}_q[t]$, 
the following theorem was proved in \cite[Theorem 30]{RMK}. We note that the theorem stated below is
slightly different from the statement of \cite[Theorem 30]{RMK}. The reason for this difference is
explained in the paragraph before Theorem \ref{thm major arcs} on page \pageref{thm major arcs}.
\begin{thm}[Theorem 30, \cite{RMK}]
\label{thm kubota 1}
Suppose $3 \leq k < p$ and $s \geq 2^k + 1$. Let $n \in \mathbb{F}_q[t]$.
Then there exists $\epsilon > 0$ such that the
following asymptotic formula holds,
\begin{equation}
\label{equation baaaaaaahhhhh1}
R_{s,k}(n) =  \mathfrak{S}_{s,k}(n)  J_{\infty}(n) q^{(s-k)P} + O\left( q^{(s-k - \epsilon)P} \right),
\end{equation}
where
\begin{equation}
\label{coeff of asympt formula}
1 \ll \mathfrak{S}_{s,k}(n) J_{\infty}(n) \ll 1.
\end{equation}
\end{thm}
Note that the implicit constants in the theorem may depend on $k$, $s$, and $q$, 
where the constant in ~(\ref{equation baaaaaaahhhhh1}) may also depend on $\epsilon$, but
they are independent of $n$ and $P$. 

We denote $\widetilde{G}_q(k)$ to be the least integer $t_0$ with
the property that, for all $s \geq t_0$, and all $n \in \mathbb{J}_q^k[t]$ with $\deg n$ sufficiently large, one has the
above asymptotic formula ~(\ref{equation baaaaaaahhhhh}).
Thus, in this language we have the following corollary as an immediate consequence of Theorem \ref{thm kubota 1}, except for the case $k=2$.
(The estimate on $\widetilde{G}_q(2)$ is treated in the paragraph after the proof of Theorem \ref{thm major arcs} on page \pageref{major k=2}.)

\begin{cor}
\label{cor kubota 1}
Suppose $2 \leq k < p$. Then we have
\begin{equation}
\notag
\widetilde{G}_q(k) \leq 
\begin{cases}
2^k + 1, & \mbox{if  } k \geq 3, \\
5, & \mbox{if  } k = 2.
\end{cases} \\
\end{equation}
\end{cor}

It is worth mentioning that one of the main advantages of
using Vinogradov-type estimates established in \cite{KLZ} or \cite{LW1} is that we can avoid
the use of Weyl differencing as the primary tool during the computation of minor arc bounds, which is the source of the restriction $k < p$
in Theorem \ref{thm kubota 1} and Corollary \ref{cor kubota 1}. Thus, via Vinogradov-type estimates we can obtain an estimate for $\widetilde{G}_q(k)$
for a larger range of $k$, which is for all $k$ not divisible by $p$.

We are now ready to state our main results. To avoid clutter in the exposition, we
present the cases $k > p$ and $ k < p$ separately. When $k > p$, as a result of our approach
we further consider three cases, $p \nmid (k-1)$, $k = p^b + 1$, and $k = m p^b + 1$, where $b, m \in\mathbb{N}$ and $p \nmid m$.
Throughout the paper, whenever we write $k = m p^b + 1$ we are assuming $b, m \in\mathbb{N}$ and $p \nmid m$,
even when these conditions are not explicitly stated.

\begin{thm}
\label{asymptotic stuff} Let $k \geq 3$ be an integer, where $p \nmid k$.
Suppose $k>p$, then we have 
\begin{eqnarray}
\notag
\widetilde{G}_q(k) \leq
\left\{
    \begin{array}{lll}
         2k \left( k - \Big{\lfloor} \frac{k}{p} \Big{\rfloor} \right) - 5 + \Big{\lfloor} \frac{ 6 \lfloor k/p \rfloor - 4  }{k-2} \Big{\rfloor},
         &\mbox{if } p \nmid (k-1),\\ \\
         4k + 5,
         &\mbox{if } k = p^b + 1,\\  \\
         \left(2 - \frac{2}{p}\right)k^2 - 2(p^b - p^{b-1} - 2)k - c_{k}, &\mbox{if } k = m p^b + 1 \mbox{ and }m>1, \\
    \end{array}
\right.
\end{eqnarray}
\newline
where $c_{k} = 2 \left( p^b - p^{b-1} - 1 - \frac{1}{p}  \right) + \Big{\lfloor} \frac{(m-1)(1-1/p)}{2} \Big{\rfloor}$.
\end{thm}

We note that when $p \nmid (k-1) $ the above theorem is proved using Lemma \ref{lemma case 1} in Section \ref{technical section},
which involves an application of the pigeon hole principle. However, when $k = m p^b + 1$
this approach is no longer effective. As a result, we have to use analogous results which rely on the large sieve inequality
instead when $m>1$, and another separate approach when $m=1$. This explains why we consider the three cases separately.

We also remark that when $k>p$ our estimates for $\widetilde{G}_q(k)$ given above
are sharper than the current available bound of $\widetilde{G}(k) \leq 2k^2 - 2k - 8 \ (k \geq 6)$ for the integer case \cite{W2}.
In particular, note that in the special case when $k = p^b + 1$ and $k>3$, we obtain a sharp linear bound
of $\widetilde{G}_q(k) \leq 4k+5$
in contrast to the quadratic bound for $\widetilde{G}(k)$.

We now state the result for the case $3 \leq k < p$.
\begin{thm}
\label{asymptotic stuff 2}
Suppose $3 \leq k < p$.
Then we have $\widetilde{G}_q(k) \leq 2k^2 - 2 \lfloor (\log k)/(\log 2) \rfloor$. 
Furthermore, $\widetilde{G}_q(7) \leq 86$ and $\widetilde{G}_q(k) \leq 2k^2 - 11$ when $k \geq 8$.
\end{thm}

We also study the slim exceptional sets associated to the asymptotic formula ~(\ref{equation baaaaaaahhhhh}).
These sets measure the frequency with which the expected formula ~(\ref{equation baaaaaaahhhhh}) does not hold.
In other words, we estimate the number of polynomials that in a certain sense do not satisfy the asymptotic formula.
For $\psi(z)$ a function of positive variable $z$, we denote by $\widetilde{E}_{s,k}(N, \psi)$
the set of $n \in I_N \cap \mathbb{J}_q^k[t]$  
for which
\begin{equation}
\label{exceptional}
\Big{|}  R_{s,k}(n) - \mathfrak{S}_{s,k}(n) J_{\infty}(n) q^{(s-k)P} \Big{|} > q^{(s-k)P} \psi(q^P)^{-1}.
\end{equation}
Note that $\widetilde{E}_{s,k}(N, \psi)$ is dependent on $q$.
We define $\widetilde{G}_q^+(k)$ to be the least
positive integer $s$ for which
$|\widetilde{E}_{s,k}(N, \psi)| = o(q^{N})$ for some
function $\psi(z)$ increasing to infinity with $z$.
We obtain the following estimates on $\widetilde{G}_q^+(k)$.
We first present the case $k > p$.
\begin{thm}
\label{exceptional stuff}
Let $k \geq 3$ be an integer, where $p \nmid k$.
Suppose $k>p$, then we have
\begin{eqnarray}
\notag
\widetilde{G}_q^+(k) \leq
\left\{
    \begin{array}{lll}
         k \left( k - \Big{\lfloor} \frac{k}{p} \Big{\rfloor} \right) - 2 + \Big{\lfloor}  \frac{3 \lfloor k/p \rfloor - 2 }{k-2} \Big{ \rfloor},
         &\mbox{if } p \nmid (k-1), \\ \\
         2k+3,
         &\mbox{if } k = p^b + 1, \\ \\
         \left(1 - \frac{1}{p}\right)k^2 - (p^b - p^{b-1} - 2)k - c'_{k}, &\mbox{if } k = m p^b + 1 \mbox{ and }m>1, \\
    \end{array}
\right.
\end{eqnarray}
\newline
where $c'_{k} = \left(  p^b - p^{b-1} - 1 - \frac{1}{p}  \right) + \Big{\lfloor} \frac{(m-1)(1-1/p)}{4} \Big{\rfloor}$.
\end{thm}
We now state the result for the case $3 \leq k < p$.
\begin{thm}
\label{exceptional stuff 2}
Suppose $3 \leq k<p$. Then we have $\widetilde{G}_q^+(k) \leq  k^2 -   \lfloor (\log k)/(\log 2) \rfloor$.
Furthermore, we have $\widetilde{G}_q^+(7) \leq 43$, and $\widetilde{G}_q^+(k) \leq k^2 - 5$ when $k \geq 8$.
\end{thm}

The organization of the rest of the paper is as follows.
In Section \ref{Prelim}, we introduce some notation and basic notions required
to carry out our discussion in the setting over $\mathbb{F}_q[t]$. In Section \ref{technical section},
we go through technical details to prove an upper bound for $\psi(\alpha, \theta)$,
which is defined in ~(\ref{def psi}). This estimate is one of the main ingredients to
obtain our minor arc estimates, for the cases $p\nmid (k-1)$ and
$k = mp^b + 1$ with $m>1$, in Section \ref{section minor arc}. We also obtain minor arc estimates for the case $k = p^b + 1$ in Section \ref{section minor arc}. We then prove a useful result related to Weyl differencing in Section \ref{section weyl}. The content of
Sections \ref{asymptotic section} and \ref{exceptional section} are similar;
we combine the material from previous sections to obtain a variant of
minor arc estimates achieved in Section \ref{section minor arc}, from which our results
follow. 

We denote $\mathbf{x} = (x_1, ..., x_{2s})$, where $x_i \in \mathbb{F}_q[t] \ (1 \leq i \leq 2s)$.
We write $N_1 \leq \ord \mathbf{x} \leq N_2$ to denote that $N_1 \leq \ord x_i \leq N_2$
for $1 \leq i \leq 2s$, and given $n_0 \in \mathbb{F}_q[t]$, we write $(\mathbf{x} - n_0)$ to denote the $2s$-tuple
$(x_1 - n_0, ..., x_{2s} - n_0)$. Confusion should not arise if the reader interprets
analogous statements in a similar manner.




\section{Preliminary}
\label{Prelim}

While the Hardy-Littlewood circle method for $\mathbb{F}_q[t]$ mirrors the classical version
familiar from applications over $\mathbb{Z}$, the substantial differences in detail between
these rings demand explanation. Our goal in the present section is to introduce notation and
basic notions that are subsequently needed to initiate discussion of key components of this version
of the circle method. The material here is taken from various sources including \cite{KLZ}, \cite{L}, \cite{LL}, \cite{LW},
and \cite{RMK}.
Associated with the polynomial ring $\mathbb{F}_q[t]$ defined over the field $\mathbb{F}_q$
is its field of fractions $\mathbb{K} = \mathbb{F}_q(t)$.
For $f/g \in \mathbb{K}$, we define an absolute value
$\langle \cdot \rangle : \mathbb{K} \rightarrow \mathbb{R}$ by
$\langle f / g \rangle = q^{\deg f - \deg g}$ (with the convention that $\deg 0 = - \infty$ and $\langle 0 \rangle = 0$).
The completion of $\mathbb{K}$ with respect to this absolute value is $\mathbb{K}_{\infty} = \mathbb{F}_q((1/t))$,
the field of formal Laurent series in $1/t$. In other words, every element $\alpha \in \mathbb{K}_{\infty} $
can be written as $\alpha = \sum_{i = -\infty}^n a_i t^i$ for some $n \in \mathbb{Z}$, coefficients $a_i = a_i (\alpha)$ in $\mathbb{F}_q \
(i \leq n)$ and $a_n \not = 0$.
For each such $\alpha \in \mathbb{K}_{\infty}$, we refer to $a_{-1}(\alpha)$ as the \textit{residue} of
$\alpha$, an element of $\mathbb{F}_q$ that we abbreviate to $\res \alpha.$
If $n < -1$, then we let $\res \alpha = 0.$
We also define the order of $\alpha$ to be $\ord \alpha = n$. Thus if $f$ is a polynomial
in $\mathbb{F}_q[t]$, then $\ord f = \deg f$. 
Note that the order on $\mathbb{K}_{\infty}$ satisfies the following property: if
$\alpha, \beta \in \mathbb{K}_{\infty}$ satisfies $\ord \alpha > \ord \beta$,
then
\begin{equation}
\label{prop of ord}
\ord (\alpha + \beta) = \ord \alpha.
\end{equation}

The field $\mathbb{K}_{\infty}$ is
a locally compact field under the topology induced by the absolute value $\langle \cdot \rangle$. 
Let $\mathbb{T} = \{ \alpha \in \mathbb{K}_{\infty} : \ord \alpha < 0 \}$.
Every element $\alpha \in \mathbb{K}_{\infty}$ can be written uniquely in the shape
$\alpha = [\alpha] + \| \alpha \|$, where the \textit{integral part} of $\alpha$ is $[\alpha] \in \mathbb{F}_q[t]$ and
the \textit{fractional part} of $\alpha$ is $\| \alpha \| \in \mathbb{T}$. 
Note that $[\cdot]$ and $\| \cdot \|$ are $\mathbb{F}_q$-linear functions on $\mathbb{K}_{\infty}$ \cite[pp.12]{RMK}. Since $\mathbb{T}$ is a compact additive subgroup of $\mathbb{K}_{\infty}$, it possesses a unique Haar measure $d \alpha.$
We normalise it, so that $\int_{\mathbb{T}} 1 \ d \alpha = 1$.
The Haar measure on $\mathbb{T}$ 
extends easily to a product measure on the $D$-fold Cartesian product
$\mathbb{T}^D$, 
for any positive integer $D$. For convenience, we will use the notation
\begin{equation}
\oint \ d \boldsymbol{\alpha} := \int_{\mathbb{T}} ... \int_{\mathbb{T}} \ d \alpha_1 ... \ d \alpha_D,
\notag
\end{equation}
where the positive integer $D$ should be clear from the context.

We are now equipped to define an analogue of the exponential function. Recall
ch$(\mathbb{F}_q) = p$. There is a non-trivial additive character
$e_q: \mathbb{F}_q \rightarrow \mathbb{C}^{\times}$
defined for each $a \in \mathbb{F}_q$ by taking $e_q(a) = \exp ( 2 \pi i \ \text{tr}(a)/p)$,
where $\text{tr}: \mathbb{F}_q \rightarrow \mathbb{F}_p$
denotes the familiar trace map.
 This character induces a map $e : \mathbb{K}_{\infty} \rightarrow \mathbb{C}^{\times}$
by defining, for each element $\alpha \in \mathbb{K}_{\infty}$, the value of $e(\alpha)$
to be $e_q(a_{-1}(\alpha))$. The orthogonality relation underlying the Fourier analysis of $\mathbb{F}_q[t]$
takes the following shape.
\begin{lem}
\label{lemma orthogonality relation}
Let $h$ be a polynomial in $\mathbb{F}_q[t]$. Then we have
\begin{eqnarray}
\label{orthogonality relation}
\int_{\mathbb{T} } e( h \alpha) \ d\alpha
=
\left\{
    \begin{array}{ll}
         0, &\mbox{if } h \in \mathbb{F}_q[t] \backslash \{ 0 \}, \\
         1, &\mbox{if } h = 0. \\
    \end{array}
\right.
\end{eqnarray}
\end{lem}
\begin{proof}
This is \cite[Lemma 1 (f)]{RMK}.
\end{proof}
The following estimate on exponential sums will be useful during the analysis in subsequent sections.
\begin{lem}
\label{lemma Kubota 7}
Let $Y \in \mathbb{N}$. Then we have
\begin{eqnarray}
\sum_{\ord x \leq Y} e(\beta x) =
\left\{
    \begin{array}{ll}
         q^{Y+1}, &\mbox{if } \ord ||{\beta}|| < -Y-1, \\
         0, &\mbox{if } \ord ||{\beta}|| \geq -Y-1. \\
    \end{array}
\right.
\end{eqnarray}
\end{lem}
\begin{proof}
This is \cite[Lemma 7]{RMK}.
\end{proof}

For each $k \geq 2$,
we define the following exponential sum
\begin{equation}
g(\alpha) = \sum_{ x \in I_X } e(\alpha x^k).
\end{equation}
Then, it is a consequence of the orthogonality relation ~(\ref{orthogonality relation}) that
\begin{equation}
\label{first integral}
R_{s,k}(n) = \int_{\mathbb{T}} g(\alpha)^s e(- n \alpha) \ d \alpha.
\end{equation}
We analyse the integral ~(\ref{first integral}) via the Hardy-Littlewood
circle method, and to this end we define sets of \textit{major} and \textit{minor arcs}
corresponding to well and poorly approximable elements of $\mathbb{T}$. Let $R_k = (k-1)X$.
Given polynomials $a$ and $g$ with $(a,g)=1$ and $g$ monic, we define the
\textit{Farey arcs} $\mathfrak{M}_k(g,a)$ about $a/g$ (associated to $k$) by
\begin{equation}
\label{def major}
\mathfrak{M}_k(g,a) = \{ \alpha \in \mathbb{T} : \ord (\alpha - a/g) < -R_k - \ord g  \}.
\end{equation}
The set of major arcs $\mathfrak{M}_k$ is defined to be the union of the sets $\mathfrak{M}_k(g,a)$
with
\begin{equation}
\label{def minor}
a,g \in \mathbb{F}_q[t], \ \  g \text{ monic}, \ \ 0 \leq \ord a < \ord g \leq X, \ \text{ and } \  (a,g) = 1.
\end{equation}
The set of minor arcs is defined to be $\mathfrak{m}_k = \mathbb{T} \backslash \mathfrak{M}_k$. It follows from
Dirichlet's approximation theorem in the setting of $\mathbb{F}_q[t]$ \cite[Lemma 3]{RMK}
that $\mathfrak{m}_k$ is the union of the sets $\mathfrak{M}_k(g,a)$
with
\begin{equation}
a,g \in \mathbb{F}_q[t], \ \  g \text{ monic}, \ \ 0 \leq \ord a < \ord g, \ \  X < \ord g \leq R_k,  \ \text{ and } \  (a,g) = 1.
\end{equation}
Notice $\mathfrak{m}_2 = \varnothing$ and for this reason, we assume $k \geq 3$
for results involving minor arcs.
We will suppress the subscript $k$ whenever there is no ambiguity with the choice of $k$ being used.
We can then rewrite ~(\ref{first integral}) as
\begin{eqnarray}
\label{second integral}
R_{s,k}(n) = \int_{\mathfrak{M}} g(\alpha)^s e(-n \alpha) \ d \alpha + \int_{\mathfrak{m}} g(\alpha)^s  e(-n \alpha) \ d \alpha,
\end{eqnarray}
and study the contribution from the major arcs and the minor arcs separately.

We have the following estimate on the major arcs, which is slightly different from what
is established in \cite{RMK}. The difference comes from our choice of $P(n)$
following \cite{LW}, instead of the approach taken in \cite{RMK}, and this choice allows us to
have a cleaner statement of the result. Applying Theorem \ref{thm major arcs} below
for the estimate of the major arcs results in the statement of Theorem \ref{thm kubota 1}
in contrast to that of \cite[Theorem 30]{RMK}.
\begin{thm}
\label{thm major arcs}
Suppose $p \nmid k$ and $s \geq 2k + 1$. 
Then there exists $\epsilon > 0$  such that
given any $n \in \mathbb{J}_q^k[t]$, the  
following asymptotic formula holds
\begin{equation}
\label{major arc thm}
\int_{ \mathfrak{M} } g(\alpha)^s \ d \alpha = \mathfrak{S}_{s,k}(n) J_{\infty}(n) q^{(s-k)P} + O\left( q^{(s-k - \epsilon )P} \right),
\end{equation}
where
$$
1 \ll \mathfrak{S}_{s,k}(n) J_{\infty}(n) \ll 1.
$$
\end{thm}
Note that the implicit constants in the theorem may depend on $s$, $q$, and $k$, 
where the constant in ~(\ref{major arc thm}) may also depend on $\epsilon$, but
they are independent of $n$ and $P$.

\begin{proof}
Let $0 < \varepsilon < 1$. Similarly as explained in the proof of \cite[Lemma 5.3]{LW},
by applying Lemma 17 of \cite{RMK} with $m = X$ and $m' = R_k + \ord g$, where $\ord g \leq \varepsilon X$, we obtain
\begin{equation}
\label{major arc estimate 1}
\int_{ \ord \beta < - R_k - \ord g} g( \beta )^s e(-n \beta) \ d \beta = J_{\infty}(n) q^{(s-k)P} + O(1),
\end{equation}
where the implicit constant may depend on $s, k, q$, and $\varepsilon$. We note that when $P$ is sufficiently large in
terms of $k$ and $\varepsilon$, it is only the cases (a) and (b) of \cite[Lemma 17]{RMK} that are relevant, and
in fact we obtain ~(\ref{major arc estimate 1}) without the $O(1)$ term. The $O(1)$ term in ~(\ref{major arc estimate 1}) comes from the small values of $P$
where this does not apply. It is also explained in the proof of \cite[Lemma 5.3]{LW} that for $s \geq k+1$, we have $1 \leq J_{\infty}(n) \ll 1$.
By \cite[Lemma 5.2]{LW}, we know that if $n \in \mathbb{J}_q^k[t]$ and $s \geq 2k+1$, then $1 \ll \mathfrak{S}_{s,k}(n) \ll 1$.

The equation ~(\ref{major arc thm}) is a consequence of ~(\ref{major arc estimate 1}) and \cite[Lemma 5.2]{LW},
and it is essentially contained in the proof of \cite[Theorem 30]{RMK}, where we replace
the use of \cite[Theorem 18]{RMK} with ~(\ref{major arc estimate 1}).
We remark that the condition $s \geq 3k+1$ is imposed in \cite[Lemma 23]{RMK}, which is also used in the proof of \cite[Theorem 30]{RMK}.
However, as stated in
\cite[pp.19]{LW} this is a result of an oversight and in fact we can relax the condition to $s \geq 2k+1$ in \cite[Lemma 23]{RMK}.
It can easily be verified that the arguments to prove ~(\ref{major arc thm}) within \cite[Theorem 30]{RMK}
also remains valid when $s \geq 2k+1$.
\end{proof}
\
When $k=2$, we know that $\mathfrak{m}_2 = \varnothing$. Hence, it follows that
\begin{equation}
\label{major k=2}
R_{s,k}(n) = \int_{ \mathfrak{M} } g(\alpha)^s \ d \alpha.
\end{equation}
Therefore, as an immediate consequence of Theorem \ref{thm major arcs}
we obtain $\widetilde{G}_q(2) \leq 5$.

It was proved in \cite[Lemma 28]{RMK} that
$\mathbb{F}_q[t] = \mathbb{J}_q^k[t]$ when $k < p$, which explains
the use of $\mathbb{F}_q[t]$ in the statement of Theorem \ref{thm kubota 1}
instead of $\mathbb{J}_q^k[t]$ as above in Theorem \ref{thm major arcs}.

Let $\mathcal{R}$ be a finite subset of $\mathbb{N}$ satisfying the following condition in \cite[pp.846]{KLZ} with $d=1$:
\begin{equation}
\label{cond* statement}
\text{Condition*: Given $l \in \mathbb{N}$, if there exists $j \in \mathcal{R}$ such that
$p \nmid {j \choose l}$, then $l \in \mathcal{R}$.}
\end{equation}
Let $J_s(\mathcal{R};X)$ denote the number of solutions of the system
\begin{equation}
\label{diag equation 1}
u_1^j + ... + u_s^j = v_1^j + ... + v_s^j  \ (j \in \mathcal{R} ),
\end{equation}
with $u_i, v_i \in I_X \ (1 \leq i \leq s)$. Since $p$ is the characteristic of $\mathbb{F}_q$,
if there exists $j, j' \in \mathcal{R}$ with $j' = p^v j$ for some $v \in \mathbb{N}$, then we
have
$$
\sum_{i=1}^s (u_i^{j'} - v_i^{j'}) = \left( \sum_{i=1}^s (u_i^{j} - v_i^{j}) \right)^{p^v}.
$$
Thus, the equations in ~(\ref{diag equation 1}) are not always independent. The absence of independence
suggests that Vinogradov-type estimates for integers cannot be adapted directly into a function field
setting. To regain independence, we instead consider
\begin{equation}
\label{def R' first}
\mathcal{R}' = \{ j \in \mathbb{N}: p \nmid j \text{ and } p^v j \in \mathcal{R} \text{ for some } v \in \mathbb{N} \cup \{0\} \}.
\end{equation}
Then we see that $J_s(\mathcal{R};X)$ also counts the number of solutions of the system
\begin{equation}
\label{diag equation 2}
u_1^j + ... + u_s^j = v_1^j + ... + v_s^j \ (j \in \mathcal{R'} ),
\end{equation}
with $u_i, v_i \in I_X \ (1 \leq i \leq s)$, or in other words $J_s(\mathcal{R};X) = J_s(\mathcal{R'};X)$.
We note here that although the equations in ~(\ref{diag equation 2}) are independent,
the set $\mathcal{R}'$ is not necessarily contained in $\mathcal{R}$.
The following theorem was proved in
\cite{LW1} and in \cite[Theorem 1.1]{KLZ} with $d=1$.
\begin{thm}[Theorem 1.1, \cite{KLZ}]
\label{eff cong result}
Suppose $\mathcal{R}$ satisfies Condition* given in ~(\ref{cond* statement}). Let $r = \card  \mathcal{R}'$, $\phi = \max_{j \in \mathcal{R}' } j$, and $\kappa = \sum_{j \in \mathcal{R}' } j$. Suppose $\phi \geq 2$ and $s \geq r \phi + r$.
Then for each $\epsilon > 0$, there exists a positive constant $C = C(s; r, \phi, \kappa; q; \epsilon)$
such that
$$
J_s(\mathcal{R};X) \leq C \left( q^X \right)^{2s - \kappa + \epsilon}.
$$
\end{thm}

The following is a useful criterion, which we utilize.
\begin{lem}
\label{Lucas}
Let $p$ be any prime
and $k = a_h p^h + ... + a_1 p + a_0$ with $0 \leq a_i < p$ $(0 \leq i \leq h)$
and $a_h \not = 0$.
The binomial coefficient ${k \choose n}$ is coprime to $p$ if and only if
$n = b_h p^h + ... + b_1 p + b_0$, where $0 \leq b_i \leq a_i$ $(0\leq i \leq h).$
\end{lem}
\begin{proof}
It follows by Lucas' Criterion \cite[pp.33]{LW} or apply \cite[Lemma A.1]{Z} with $d=1$.
\end{proof}
As a consequence of Lemma \ref{Lucas}, we have the following lemma.
\begin{lem}
\label{lemma j_0}
Let $p$ be any prime. Suppose $k = m p^b + 1$ with $m, b \in \mathbb{N}$ and $p \nmid m$.
Then, $(k - p^b)$ is the largest number less than $(k-1)$ such that
${k \choose k - p^b} \not \equiv 0 \ (\text{mod } p)$.
\end{lem}
\begin{proof}
Let $m = c_a p^a + c_{a-1} p^{a-1} + ... + c_1 p + c_0$ with $0 \leq c_i <p$ and $ 0 < c_0$.
Thus, we have $k = c_a p^{a+b} + c_{a-1} p^{a-1+b} + ... + c_1 p^{b+1} + c_0 p^b + 1$.
For $1 \leq j \leq p^b$, write $k-j =
c_a p^{a+b} + c_{a-1} p^{a-1+b} + ... + c_1 p^{b+1} + d_b p^b
+ d_{b-1} p^{b-1} + ... + d_1 p + d_0$
with $0 \leq d_i < p$. Then, by Lemma \ref{Lucas}, ${k \choose k - j} \not \equiv 0 \ (\text{mod } p)$
if and only if $d_b \leq c_0$, $d_i = 0$ $(1 \leq i < b)$ and  $d_0 \leq 1$.
Therefore, it is not too difficult to verify that
${k \choose k-j} \not \equiv 0 \ (\text{mod } p)$ only when $j = 1$ and $p^b$
in the range $1 \leq j \leq p^b$.
\end{proof}

For a prime $p =$ ch$(\mathbb{F}_q)$ and  $k \in \mathbb{N}$ with $p \nmid k$, we define
$j_0(k,q) = j_0$ to be
\begin{equation}
\label{def j_0}
j_0 := \max_{0 < j < k} \left\{ j : p \nmid j \text{ and }{k \choose j} \not \equiv 0 \ (\text{mod } p) \right\}.
\end{equation}
If $p \nmid (k-1)$, then $j_0 = k-1$.
On the other hand, if $k = m p^b + 1$ for some $m, b \in \mathbb{N}$ and $p \nmid m$, then
$j_0 = k - p^b$ by Lemma \ref{lemma j_0}. We record the values of $j_0$ here for reference,
\begin{eqnarray}
\label{value of j_0}
j_0 =
\left\{
    \begin{array}{ll}
         k - 1, &\mbox{if } p \nmid (k-1), \\
         (m - 1) p^b + 1, &\mbox{if } k = m p^b + 1. \\
    \end{array}
\right.
\end{eqnarray}
With application of Theorem \ref{eff cong result} 
in mind, we define the following two
sets
\begin{equation}
\label{def R}
\mathcal{R} = \{1, 2, ... , j_0 , k \} \cup \{ k-1 \}
\end{equation}
and
\begin{eqnarray}
\label{def R'}
\mathcal{R}' &=& \{ j \in \mathbb{N}: p \nmid j \text{ and } p^v j \in \mathcal{R} \text{ for some } v \in \mathbb{N} \cup \{0\} \}
\\
&=& \{ j : j \in \mathcal{R} \text{ and } p \nmid j \}.
\notag
\end{eqnarray}
The first equality is the definition of $\mathcal{R}'$, which comes from ~(\ref{def R' first}),
but the second equality requires a slight justification.
If $p \nmid (k-1)$, then $\mathcal{R} = \{1, 2, ... , k \}$
and the second equality of ~(\ref{def R'}) is immediate.
If $k = m p^b+1$, then $k \in \mathcal{R}'$. We also have $k-1 = m p^b \not \in \mathcal{R}'$
and $m \in \mathcal{R}'$. However, since $j_0 = (m-1)p^b + 1 > m$ and $p \nmid m$,
it follows that $\mathcal{R}' = \{ j : 1 \leq j \leq j_0  \text{ and } p \nmid j \} \cup \{ k \}$
from which we obtain the second equality of ~(\ref{def R'}).

We let $\card \mathcal{R}' = r$
and let $\mathcal{R}' = \{t_1, ..., t_r\}$, where $t_1 < ...< t_r$.
Clearly, we have $t_r = k$ and it follows by our definition of $j_0$ and $\mathcal{R}$ that $t_{r-1} = j_0$.
We can verify by simple calculation that
\begin{eqnarray}
\label{defn of r}
r =
\left\{
    \begin{array}{ll}
         k - \lfloor k/p \rfloor, &\mbox{if } p \nmid (k-1), \\
         (1 - 1/p)(k-p^b) + (1 + 1/p), &\mbox{if } k = m p^b + 1. \\
    \end{array}
\right.
\end{eqnarray}
In particular, if $k = p^b + 1$, then $r=2$.
For the remainder of the paper, whenever we refer to $\mathcal{R}$, $\mathcal{R}'$
and $r$, we mean ~(\ref{def R}), ~(\ref{def R'}), and ~(\ref{defn of r}), respectively.
\begin{lem}
\label{condition*}
$\mathcal{R}$ satisfies Condition* given in ~(\ref{cond* statement}).
\end{lem}
\begin{proof}
If $p \nmid (k-1)$, then $\mathcal{R} = \{ 1, 2, ..., k \}$
and it satisfies Condition*. This is easy to see, because
suppose for some $l \in \mathbb{N}$, there exists
$j \in \mathcal{R}$ such that $p \nmid {j \choose l}$. Then we have
$1 \leq l \leq j \leq k$, and hence $l \in \mathcal{R}$.
On the other hand, if $k = m p^b + 1$, then we have
$\mathcal{R} = \{ 1, 2, ..., j_0, k-1, k \}$.
Suppose we are given some $l \in \mathbb{N}$. It is clear that
if $l > k$, then there does not exist $j \in \mathcal{R}$ such that
$p \nmid {j \choose l}$, because $ {j \choose l} = 0$. Thus, it suffices to show that
for $j_0 < l < (k-1)$, ${j \choose l}\equiv 0 \ (\text{mod } p)$ for all $j \in \mathcal{R}$.
Clearly, ${j \choose l}\equiv 0 \ (\text{mod } p)$ for $j \leq j_0$.
Lemma \ref{lemma j_0} gives us that
${k \choose l} \equiv 0 \ (\text{mod } p)$ for $j_0 < l < (k-1)$.
Therefore, we only need to verify ${k-1 \choose l} \equiv 0 \ (\text{mod } p)$
for $j_0 = (k - p^b) < l < (k-1)$. 
Every $l$ in this range can be written as
$l = (m-1) p^b + c_{b-1} p^{b-1} + ... + c_1 p + c_0$,
where $0 \leq c_i < p$. Since $(k-1) = m p^b$, by Lemma \ref{Lucas} we have
${ k-1 \choose l}   
\not \equiv 0 \ (\text{mod } p)$
if and only if $c_i = 0$ for $0 \leq i < b$, or in other words $l = (m-1) p^b = k-p^b-1$.
Because $l = k-p^b-1$ is not in the range of $l$ we are considering, 
it follows that $\mathcal{R}$ satisfies Condition*.
\end{proof}

\section{Technical Lemmas}
\label{technical section}

We will be applying the following large sieve inequality in this section. Given a set
$\Gamma \subseteq \mathbb{K}_{\infty}$, if for any distinct elements $\gamma_1, \gamma_2 \in \Gamma$
we have $\ord (\gamma_1 - \gamma_2) > \delta$, then we say the points $\{ \gamma: \gamma \in \Gamma \}$
are \textit{spaced at least} $q^{\delta}$ \textit{apart in} $\mathbb{T}$.

\begin{thm}[Theorem 2.4, \cite{H}]
\label{large sieve}
Given $A,Z \in \mathbb{Z}^+$, let $\Gamma \subseteq \mathbb{K}_{\infty}$ be a set whose elements are
spaced at least $q^{-A}$ apart in $\mathbb{T}$. Let $(c_x)_{x \in \mathbb{F}_q[t]}$ be a sequence of
complex numbers. For $\alpha \in \mathbb{K}_{\infty}$, define
$$
\mathcal{S}(\alpha) = \sum_{\ord x \leq Z} c_x e(x \alpha).
$$
Then we have
$$
\sum_{\gamma \in \Gamma} |\mathcal{S}(\gamma)|^2 < \max \{ q^{Z+1}, q^{A-1} \} \sum_{\ord x \leq Z} |c_x|^2.
$$
\end{thm}

Recall $I_X:=\{ x \in \mathbb{F}_q[t] : \ord x < X \}$.
Let $k \geq 3$, $\theta \in \mathfrak{m}_{k}$, $0 \not = c \in \mathbb{F}_q$, $\alpha \in \mathbb{T}$, and $j_0$ be as defined in Section \ref{Prelim}.
In this section, we find an upper bound for the following exponential sum,
\begin{equation}
\label{def psi}
\psi(\theta , \alpha) = q^{-X} \sum_{ y \in I_X} \sum_{\ord h \leq j_0 (X-1) } e(- chy^{k-j_0} \theta - \alpha h).
\end{equation}
The estimates obtained for $\psi(\theta , \alpha)$ is one of our main ingredients for
computing the minor arc estimates in Section \ref{section minor arc}.
To achieve this goal, the precise value of $j_0$ with respect to $k$
plays an important role. Hence, we consider the following two cases
separately: $p \nmid (k-1)$ and $k = m p^b + 1$ with $m, b \in \mathbb{N}$, $m>1$, and $p \nmid m$.
We do not consider the case $k = p^b+1$ here, because we apply a different method to bound the minor arcs
in this case.

First, we make several observations, which we use throughout this section.
Let $\theta = a/g + \beta$, where $(a,g)=1$.
Let $x$, $y \in I_X$ and
$x \not = y$.
Then, since $\| \cdot \|$ is $\mathbb{F}_q$-linear, 
we have
\begin{eqnarray}
\label{eqn pr 1}
&&\ord ( \| c x^{k-j_0} \theta  + \alpha \| - \|c y^{k-j_0} \theta  + \alpha \| )
\\
\notag
&=& \ord \| (x^{k-j_0} - y^{k-j_0}) \theta  \|
\\
\notag
&=& \ord ( \| (x^{k-j_0}-y^{k-j_0}) a/g \| + \| (x^{k-j_0}-y^{k-j_0}) \beta  \| ).
\end{eqnarray}
Since $\mathbb{F}_q[t]$ is a unique factorization domain,
we have $(x^{k-j_0} - y^{k-j_0}) a \not = 0$ as long as $a \not = 0$.
Note that it is possible to get $a = 0$, when $\ord g = 0.$

Suppose $(x^{k-j_0} - y^{k-j_0})a/g \in \mathbb{F}_q[t]$. Then, we have $ \| (x^{k-j_0}-y^{k-j_0}) a/g \| = 0$
and
\begin{eqnarray}
\label{eqn pr 2}
\ord ( \| c x^{k-j_0} \theta  + \alpha \| - \|c y^{k-j_0} \theta  + \alpha \| )
= \ord \| (x^{k-j_0}-y^{k-j_0}) \beta  \| .
\end{eqnarray}

On the other hand, if $(x^{k-j_0} - y^{k-j_0})a/g \not \in \mathbb{F}_q[t]$, write
$$
\frac{a}{g} (x^{k-j_0} - y^{k-j_0}) = s_0 + a_{-j}t^{-j} + a_{-j-1}t^{-j-1} + ...
$$
with $s_0 \in \mathbb{F}_q[t]$, $a_i \in \mathbb{F}_q$ for $i \leq -j \leq -1$ and $a_{-j} \not = 0$.
Here we know such $a_{-j} \not = 0$ exists, because
$(x^{k-j_0} - y^{k-j_0})a/g \not \in \mathbb{F}_q[t]$.
Then it follows that
$$
a (x^{k-j_0} - y^{k-j_0}) - g s_0 = g (a_{-j}t^{-j} + a_{-j-1}t^{-j-1} + ... \ ).
$$
Since the left hand side is a polynomial, we have $- j + \ord g \geq 0$. Consequently,
we obtain
\begin{equation}
\label{eqn pr 3}
 \ord \| (x^{k-j_0} - y^{k-j_0})a/g \| \geq - \ord g.
\end{equation}

\subsection{Case $p \nmid (k-1)$}
Here we have $j_0 = k-1$, or equivalently $k - j_0 =1$. In this situation,
we obtain an upper bound for $\psi(\theta, \alpha)$
in a way analogous to the case for integers in \cite{W1}.
We have the following lemma.

\begin{lem}
\label{lemma case 1}
Suppose $k \geq 3$, $p \nmid k$, and $p \nmid (k-1)$. Let $\theta \in \mathfrak{m}_{k}$ and $\alpha \in \mathbb{T}$.
Then we have
$$
\psi(\theta , \alpha) \leq q^{(j_0 - 1)X}.
$$
\end{lem}
\begin{proof}
Let $\theta = a/g + \beta \in \mathfrak{M}_{k}(g,a) \subseteq \mathfrak{m}_{k}$.
Let $x$, $y \in I_X$ and
$x \not = y$.
Then, we know
$(x^{k-j_0} - y^{k-j_0})a/g \not \in \mathbb{F}_q[t]$,
because $k - j_0 = 1$ and $\ord g > X$. Consequently, we have ~(\ref{eqn pr 3}).
Recall $R_k = (k-1)X$. For simplicity we let $R = R_k$.
Since $R 
> (k-j_0) (X-1)$ and $\ord \beta < (- R -  \ord g)$,
we have
$$
\ord (x^{k-j_0}-y^{k-j_0}) \beta < (k-j_0)(X-1) - R - \ord g  < - \ord g \leq 0.
$$
Thus, 
we obtain from ~(\ref{prop of ord}) and ~(\ref{eqn pr 1})
\begin{eqnarray}
\label{eqn i 4}
\ord \left( \|  c x^{k-j_0} \theta  + \alpha \| - \|  c y^{k-j_0} \theta + \alpha \| \right)
&=&
\ord \|  (x^{k-j_0}-y^{k-j_0}) a/g \|
\\
\notag
&\geq& - \ord g
\notag
\\
&\geq& -R.
\notag
\end{eqnarray}

Suppose there exists $y \in I_X$
such that
$\ord \| c y^{k - j_0} \theta + \alpha \| < (-j_0 (X-1) -1)$, or equivalently,
\begin{equation}
\label{bdd ord 1}
\ord \| c y \theta + \alpha \| < - (k-1) (X-1) - 1.
\end{equation}
This means
the first $\left( (k-1) (X-1) + 1 \right)$ coefficients of $\| c y \theta + \alpha \|$
are $0$. Hence, it takes the form
$$
\|  c y \theta + \alpha \| = 0 \ t^{-1} + 0 \ t^{-2} + ...
\ + 0 \ t^{-(k-1) (X-1)-1} + a_{- (k-1) (X-1) -2 } t^{-(k-1) (X-1)-2} + ... \ + a_{-R}t^{-R} + ... \ .
$$
Note that there are only $q^{k-2}$ possibilities for the $(k-2)$-tuple
$(a_{-(k-1)(X-1) -2}, ... \ , a_{-R})$. Thus, if there are more than
$q^{k-2}$ such polynomials $y \in I_X$ satisfying ~(\ref{bdd ord 1}),
then by the pigeon hole principle there exists a pair $x$ and $y$ in $I_X$ for which
the first $R$ coefficients of $\| c x \theta + \alpha \|$ and $\| c y \theta + \alpha \|$
agree. However, this contradicts ~(\ref{eqn i 4}). Therefore, it follows by ~(\ref{def psi}) and Lemma \ref{lemma Kubota 7} that
$$
\psi(\theta , \alpha) \leq  q^{-X + k - 2 + (k-1)(X-1)+1} = q^{(k-2)X}.
$$
\end{proof}

\subsection{Case $k = m p^b + 1$ with $m > 1$}
Here we have $j_0 = k - p^b > p^b = k - j_0$.
When $p \nmid (k-1) $, we had that the difference between $j_0 (X-1) +1$ and $R_k = (k-1)X$
was small enough compared to $X$ - in fact it was constant with respect to $X$ - which was the reason our application of the pigeon hole
principle was effective in Lemma \ref{lemma case 1}. However, when $k = m p^b + 1$ this is no longer the case
as $R_k - j_0(X-1) - 1 = (p^b - 1)X + (k -p^b -1)$. 

It follows from the definition of the major arcs that $\mathfrak{M}_{k} \subseteq \mathfrak{M}_{k - j_0 + 1}$,
hence $\mathfrak{m}_{k - j_0 +1} \subseteq \mathfrak{m}_{k}$.
Therefore, given $\theta \in \mathfrak{m}_{k}$, we have
either $\theta \in \mathfrak{m}_{k - j_0 +1}$ or $\theta \in \mathfrak{M}_{k - j_0+1}$.
We consider these two cases separately in Lemmas \ref{lem 7} and \ref{lem 8}.
The argument in Lemma \ref{lem 7} is similar to that of Lemma \ref{lemma case 1}.
However, in Lemma  \ref{lem 8} we use a different approach, which
relies on the large sieve inequality given in Theorem \ref{large sieve} instead.


\begin{lem}
\label{lem 7}
Let $k = m p^b +1$ with $m>1$ and $\theta \in \mathfrak{m}_{k}$. Suppose $\theta \in \mathfrak{m}_{k - j_0+1}.$
Then we have
$$
\psi(\theta , \alpha) \ll q^{(j_0 - 1/p^b ) X },
$$
where the implicit constant depends only on $q$ and $k$.
\end{lem}

\begin{proof} Let $\theta = a/g + \beta \in \mathfrak{M}_{k - j_0 + 1} (g,a) \subseteq \mathfrak{m}_{k - j_0+1}$, and we know $R' \geq \ord g > X$,
where $R' = R_{k - j_0 + 1} = (k - j_0)X$.
Given $y \in I_X$, it takes the form
\begin{equation}
\label{form of y 3}
y = c_{X-1} t^{X-1} + ... + c_{ \lfloor X/p^b \rfloor }t^{ \lfloor X/p^b \rfloor } + ... + c_0.
\end{equation}
Let $L = (X - \lfloor X/p^b \rfloor)$.
Order the $L$-tuples of elements of $\mathbb{F}_q$ in any way, for example, we may take one bijection between
$\mathbb{F}_q$ and $\{1, ..., q\}$, and use the lexicographic ordering on $(\mathbb{F}_q)^L$. We can then split $I_X$ into
$q^L$ subsets $T_1, T_2, ..., T_{q^L}$, where
\begin{eqnarray}
\notag
T_l = \{ y \in I_X: \text{ given $y$ in the form ~(\ref{form of y 3}), the coefficients }
\left( c_{X-1}, ..., c_{ \lfloor X/p^b \rfloor } \right) \phantom{1234567891}
\\
\notag
\text{ is exactly the }l \text{-th }  L\text{-tuple}\}.
\end{eqnarray}

Then, we have for some $T' = T_l$
\begin{equation}
\label{eqn 2 a 2}
\psi(\theta , \alpha) \ll q^{-X + X -  X/p^b } \Big{|} \sum_{y \in T'} \sum_{\text{ord }h \leq j_0 (X-1) } e(- chy^{k-j_0} \theta - \alpha h) \Big{|}.
\end{equation}

Given any distinct $x, y\in T'$, we have
$$
\ord  (x^{k-j_0} - y^{k-j_0}) = \ord  (x-y)^{p^b} \leq X,
$$
and hence,
$(x^{k-j_0} - y^{k-j_0})a/g \not \in \mathbb{F}_q[t]$.
Thus, by ~(\ref{eqn pr 3}) we have
$ \ord  \| (x^{k-j_0} - y^{k-j_0})a/g \| \geq - \ord  g$.
Since $\ord \beta < - R' -  \ord g $ and $R' = (k-j_0)X > X$,
we have $ \ord  (x^{k-j_0}-y^{k-j_0}) \beta < X - R' - \ord g  < - \ord g \leq 0$.
Therefore, by ~(\ref{prop of ord}) and ~(\ref{eqn pr 1}), we obtain
\begin{equation}
\label{eqn 2 a 1}
\ord  \left( \|  c x^{k-j_0} \theta  + \alpha \| - \|  c y^{k-j_0} \theta + \alpha \| \right) \geq - \text{ord }g \geq -R'.
\end{equation}

Suppose there exists $y \in T'$ such that
$\ord  \| c y^{k - j_0} \theta + \alpha \| < -j_0 (X-1) -1$.
This means
the first $j_0 (X-1) + 1$ coefficients of $\| c y^{k - j_0} \theta + \alpha \|$
must be $0$, or in other words it takes the form
$$
\|  c y^{k - j_0} \theta + \alpha \| = 0 \ t^{-1} + 0 \ t^{-2} + ...
\ + 0 \ t^{-j_0 (X-1)-1} + a_{- j_0 (X-1) -2 } t^{-j_0 (X-1)-2} + ...
$$
If there is another distinct $x \in T'$, which satisfies the same
condition, then the first $j_0 (X-1) + 1$ coefficients of
$\| c x^{k - j_0} \theta + \alpha \|$ agree with that
of $\| c y^{k - j_0} \theta + \alpha \|$.
However, this contradicts ~(\ref{eqn 2 a 1})
as $R' = (k-j_0)X < j_0(X-1) + 1$ for $X$ sufficiently large.
Hence, there is at most one such $y$. Therefore,
it follows by ~(\ref{eqn 2 a 2}) and Lemma \ref{lemma Kubota 7} that
$$
\psi(\theta , \alpha) \ll  q^{-X + X -  X/p^b  + j_0 (X-1) + 1 } \ll q^{(j_0 - 1/p^b ) X }.
$$
\end{proof}

\begin{lem}
\label{lem 8}
Let $k = m p^b +1$ with $m>1$ and $\theta \in \mathfrak{m}_{k}$. Suppose $\theta \in \mathfrak{M}_{k - j_0+1}$.
Then we have
$$
\psi(\theta , \alpha) \ll  q^{(j_0 - 1/(4p^b)) X },
$$
where the implicit constant depends only on $q$.
\end{lem}
\begin{proof}
Let $\theta = a/g + \beta \in \mathfrak{M}_{k - j_0 + 1} (g,a)
 \subseteq \mathfrak{M}_{k - j_0+1}$.
Then, we have
$\ord g \leq X$
and
\begin{equation}
\label{lower bound beta 2.3}
-R_k - \ord  g \leq \ord \beta  < -R_{k - j_0 + 1} - \ord  g,
\end{equation}
where $R_k = (k-1)X$ and $R_{k - j_0 + 1} = (k - j_0)X$. For simplicity, we denote
$R = R_k$ and $R' = R_{k - j_0 + 1}$.
We have the above lower bound, for otherwise it would mean $\theta \in \mathfrak{M}_{k}$.

By the Cauchy-Schwartz inequality, we obtain
\begin{eqnarray}
\label{C-S iib}
\psi(\theta , \alpha)
\ll
q^{-X} q^{X/2}
S^{1/2},
\end{eqnarray}
where
$$
S = \sum_{y \in I_X} \Big{|} \sum_{\ord h \leq j_0 (X-1) } e(- chy^{k-j_0} \theta - \alpha h) \Big{|}^2.
$$

Let $\delta' > 0$ be sufficiently small, and in particular we make sure $\delta' \leq 1$.
We consider two cases: $\ord g > \delta' X $ and $\ord g \leq \delta' X $.

Case 1: Suppose $\ord g > \delta' X $.
Given $y \in I_X$, it takes the form
\begin{equation}
\label{form of y 1}
y = c_{X-1} t^{X-1} + ... + c_{ \lfloor \delta' X/p^b \rfloor }t^{ \lfloor \delta' X/p^b \rfloor } + ... + c_0.
\end{equation}
Let $L = X - \lfloor \delta' X/p^b \rfloor$.
Order the $L$-tuples of elements of $\mathbb{F}_q$ in any way. 
We can then split $I_X$ into
$q^{L}$ subsets, $T_1, T_2, ..., T_{q^L}$, where
\begin{eqnarray}
\notag
T_l = \{ y \in I_X: \text{given $y$ in the form ~(\ref{form of y 1}), the coefficients }
\left( c_{X-1}, ..., c_{ \lfloor \delta' X/p^b \rfloor } \right) \phantom{1234567891}
\\
\notag
\text{  is exactly the }l \text{-th }  L \text{-tuple}\}.
\end{eqnarray}

Then we have for some $T' = T_l$
\begin{equation}
\label{eqn ii b 1}
S \ll q^{ X - \delta' X/p^b } \sum_{y \in T'} \Big{|} \sum_{\ord h \leq j_0 (X-1) } e(- chy^{k-j_0} \theta - \alpha h)\Big{|}^2.
\end{equation}

Recall $k - j_0 = p^b$. Given any $x, y\in T'$,
we have
$$
\ord \left( x^{k-j_0} - y^{k-j_0} \right) = \ord (x-y)^{p^b} \leq \delta' X,
$$
and hence,
$(x^{k-j_0} - y^{k-j_0})a/g \not \in \mathbb{F}_q[t]$.
Thus, we have
$ \ord \| (x^{k-j_0} - y^{k-j_0})a/g \| \geq - \ord g$ by ~(\ref{eqn pr 3}).
Since $\ord \beta < (- R' -  \ord g)$ and $R' = X > \delta' X$,
we have $$ \ord (x^{k-j_0}-y^{k-j_0}) \beta < \delta' X - R' - \ord g  < - \ord g \leq 0.$$
Therefore, by ~(\ref{prop of ord}) and ~(\ref{eqn pr 1}), we obtain
\begin{equation}
\label{eqn ii b 2}
\ord \left( \| ( c x^{k-j_0} \theta  + \alpha) \| - \| ( c y^{k-j_0} \theta + \alpha) \| \right) \geq - \ord g \geq -X.
\end{equation}

Since
$ \max \{X, j_0 (X-1)+1 \}
\leq j_0 X$, we have by Theorem \ref{large sieve} 
\begin{equation}
\label{eqn iiib 3}
S \ll q^{ X -  \delta' X/p^b } \sum_{y \in T'} \Big{|} \sum_{\ord h \leq j_0 (X-1) } e(- chy^{k-j_0} \theta - \alpha h) \Big{|}^2
\ll
q^{ X -  \delta' X/p^b } q^{ 2 j_0 X}.
\end{equation}

Case 2: Suppose $\ord g \leq \delta' X $. Let $\epsilon > 0$ be sufficiently small.
We order the polynomials of degree less than
$L' = \lceil (1- \epsilon)X \rceil$ in any way, and call them $p_1, p_2, ..., p_{q^{L'}}$.
We then split $I_X$ into $q^{L'}$ subsets, $T_1, T_2, ..., T_{q^{L'}}$,
where given any $x \in T_l$, $1 \leq l \leq q^{L'}$,
the coefficients of $x$ for powers less than
$L'$ agree with that of $p_l$.
Thus, we have for some $T' = T_l$
\begin{equation}
\label{eqn iib 4}
S \ll q^{(1 - \epsilon) X} \sum_{y \in T'} \Big{|} \sum_{\ord h \leq j_0 (X-1) } e(- chy^{k-j_0} \theta - \alpha h) \Big{|}^2.
\end{equation}

Given any $x$, $y \in T'$ with
$x^{k - j_0} \not \equiv y^{k-j_0} (\text{mod }g)$, we have
$(x^{k-j_0} - y^{k-j_0})a/g \not \in \mathbb{F}_q[t]$.
Thus, by ~(\ref{eqn pr 3}) we have
$ \ord \| (x^{k-j_0} - y^{k-j_0})a/g \| \geq - \ord  g$.
Since $\ord  \beta < - R' -  \ord g$ and $R' = (k-j_0)X > ( k-j_0) (X-1)$,
we have
\begin{equation}
\label{bdd ord 2}
\ord  (x^{k-j_0}-y^{k-j_0}) \beta < (k-j_0)(X-1) - R' - \ord g  < - \ord g \leq 0.
\end{equation}
Therefore, by ~(\ref{prop of ord}) and ~(\ref{eqn pr 1}),
we obtain
\begin{equation}
\label{eqn iib 5}
\ord  \left( \| ( c x^{k-j_0} \theta  + \alpha) \| - \| ( c y^{k-j_0} \theta + \alpha) \| \right) \geq - \ord g \geq - \delta' X.
\end{equation}

On the other hand, suppose we have distinct $x,y \in T'$, where
$x^{k-j_0}  \equiv y^{k-j_0} (\text{mod }g)$. Then we have $(x^{k-j_0} - y^{k-j_0})a/g \in \mathbb{F}_q[t]$ from which ~(\ref{eqn pr 2}) follows.
Also, because $x,y \in T'$ and $k - j_0 = p^b$, we obtain
$$
\ord (x^{k-j_0} - y^{k - j_0}) = \ord (x-y)^{p^b} \geq p^b L'.
$$
Therefore, it follows by ~(\ref{eqn pr 2}), ~(\ref{lower bound beta 2.3}) and ~(\ref{bdd ord 2}),
\begin{eqnarray}
\ord  \left( \| ( c x^{k-j_0} \theta  + \alpha) \| - \| ( c y^{k-j_0} \theta + \alpha) \| \right)
&=& \ord   (x^{k-j_0} - y^{k-j_0}) \beta
\\
\notag
&\geq&
p^b L' - R - \ord g
\\
\notag
&\geq&
(k-j_0)(1 - \epsilon)X - (k-1)X - \delta' X
\\
\notag
&=&
-j_0 X + (1 - (k-j_0)\epsilon - \delta' )X
\\
\notag
&\geq&
-j_0 X.
\end{eqnarray}

Since
$ \max \{\delta' X, j_0 X, j_0 (X-1)+1 \}
\leq j_0 X$, we have by Theorem \ref{large sieve}
\begin{equation}
\label{eqn iib 6}
S \ll q^{(1 - \epsilon) X} \sum_{y \in T' } \Big{|} \sum_{\ord h \leq j_0 (X-1) } e(- chy^{k-j_0} \theta - \alpha h) \Big{|}^2
\ll
q^{(1 - \epsilon) X} q^{2 j_0 X}.
\end{equation}


Note that the only restrictions we had so far for $\delta'$ and $\epsilon$
were: $0 < \delta' \leq 1$, $0 < \epsilon $, and
\begin{equation}
\label{condition 1 case3}
0 \leq 1 - (k-j_0) \epsilon - \delta'.
\end{equation}
We have by ~(\ref{eqn iiib 3}) and (\ref{eqn iib 6})
$$
S \ll q^{ X - \delta' X/p^b } q^{ 2 j_0 X} + q^{(1 - \epsilon) X} q^{2 j_0 X}.
$$
In order to minimize the right hand side of the above inequality, we set $\epsilon = \delta' /p^b$.
Then, since $k - j_0 = p^b$, ~(\ref{condition 1 case3}) can be simplified to
$$
2 \delta' \leq 1.
$$
By letting $\delta' = 1/2$,
we obtain by ~(\ref{C-S iib})
$$
\psi(\theta , \alpha) \ll
q^{-X/2} S^{1/2}
\ll q^{\left(j_0 - \delta \right) X},
$$
where $\delta = 1/(4p^b)$.

\end{proof}


\section{A bound on the minor arcs}
\label{section minor arc}

We obtain estimates on the minor arcs in this section.
In Section \ref{minor arc case 1}, we give bounds on the minor arcs when $p \nmid (k-1)$ and $p = mk^b+1$, $m>1$.
The remaining case when $k = p^b + 1$ requires a different approach, and it is treated separately in Section \ref{minor arc case 2}.
The reason we require a different approach is that when $k = p^b+1$, the method in Section \ref{minor arc case 1} results in
an exponential sum that is more complicated to estimate than $\psi(\alpha, \theta)$. 
Thus we take a more basic approach in this case.

\subsection{Cases $p \nmid (k-1)$ and $p = mk^b+1$, $m>1$}
\label{minor arc case 1}

Let $\mathcal{R}'$ be as defined in ~(\ref{def R'}).
Recall from the paragraph after Lemma \ref{condition*} that $\card \mathcal{R}' = r$,
and $t_1, ..., t_r$ are the elements of
$\mathcal{R}'$ in increasing order. The main results of this section are the
following estimates on the minor arcs.
\begin{thm}
\label{Thm minor arc bound 1}
Suppose $k \geq 3$ and $p \nmid k$. Suppose further that either $p \nmid (k-1)$ or $p = m k^b+1$, $m>1$. Let $\kappa = \sum_{j=1}^r t_j$, where $\mathcal{R}' = \{ t_1, ..., t_r\}$ and $t_j \leq t_{j+1}$.
Let
\begin{equation}
\notag
\delta_0 =
\begin{cases}
1, & \mbox{if  } p \nmid (k-1), \\
\frac{1}{4p^b}, & \mbox{if  } k = mp^b + 1, m > 1. \\
\end{cases} \\
\end{equation}
Then we have
$$
\int_{\mathfrak{m}} |g(\alpha)|^{2s} \ d \alpha
\ll
q^{(\kappa - k - \delta_0)X} J_s(\mathcal{R}',X),
$$
where the implicit constant depends only on $q$ and $k$.
\end{thm}

Recall from above that if $p \nmid (k-1)$, then $r = k - \lfloor k/p \rfloor$.
On the other hand, if $k = m p^b + 1$, then $r = (1 - 1/p)(k-p^b) + (1 + 1/p).$ 

\begin{cor}
\label{minor arc bound 2}
Suppose $k \geq 3$, $p \nmid k$ and $s \geq (rk+r)$. Suppose further that either $p \nmid (k-1)$ or $p = m k^b+1$, $m>1$. Let $\delta_0$ be as in the statement of Theorem \ref{Thm minor arc bound 1}.
Then for each $\epsilon > 0$, we have
$$
\int_{\mathfrak{m}} |g(\alpha)|^{2s} \ d \alpha
\ll
q^{(2s - k - \delta_0 + \epsilon)X},
$$
where the implicit constant depends only on $s, q, k, \mathcal{R}'$, and $\epsilon$.
\end{cor}

\begin{proof}
This is an immediate consequence of
applying Theorem ~\ref{eff cong result}
to Theorem ~\ref{Thm minor arc bound 1}.
\end{proof}

Before we begin with our proof of Theorem \ref{Thm minor arc bound 1}, we set some notation.
First we define the following exponential sums:
\begin{equation}
f(\boldsymbol{\alpha} ) = \sum_{x \in I_X } e \left( \sum_{j=1}^{r-1} \alpha_{t_j} x^{t_j} + \alpha_k x^k \right),
\end{equation}
and
\begin{equation}
F(\boldsymbol{\beta}, \theta) = \sum_{x \in I_X } e \left( \sum_{j=1}^{r-2} \beta_{t_j} x^{t_j} + \theta x^k \right).
\end{equation}
We will also use the notation $f(\boldsymbol{\alpha}, \theta)$ to mean
$$
f(\boldsymbol{\alpha}, \theta) = f(\alpha_{t_1}, \alpha_{t_2}, ..., \ \alpha_{t_{r-1}}, \theta ).
$$

We also define for $1 \leq j \leq k$,
\begin{equation}
\sigma_{s,j} (\mathbf{x}) = \sum_{i=1}^{s}(x_i^j - x_{s+i}^j).
\end{equation}

Recall
$J_s(\mathcal{R}',X)$ is the number of solutions of the system
$$
u_1^j + ... + u_s^j = v_1^j + ... + v_s^j  \ \ (j \in \mathcal{R}')
$$
with $u_j,v_j \in I_X \ (1 \leq j \leq s)$.
By the orthogonality relation ~(\ref{orthogonality relation}), it follows that
\begin{equation}
\label{def of J}
J_s(\mathcal{R}',X) = \oint |f(\boldsymbol{\alpha})|^{2s} \ d \boldsymbol{\alpha}.
\end{equation}

\begin{proof}[Proof of Theorem \ref{Thm minor arc bound 1}] We begin by expressing the mean value of $g(\alpha)$
in terms of mean value of $F(\boldsymbol{\beta}, \theta)$. Since $ \overline{F(\boldsymbol{\beta}, \theta) } = F(- \boldsymbol{\beta}, - \theta)$,
we see that
\begin{eqnarray}
|F(\boldsymbol{\beta}, \theta)|^{2s} &=&
\prod_{i=1}^s \left( \sum_{x_i, x_{s+i} \in I_X } e \left( \sum_{j=1}^{r-2} \beta_{t_j} (x_i^{t_j} - x_{s+i}^{t_j} )
+ \theta (x_i^k - x_{s+i}^k) \right) \right)
\notag
\\
&=& \sum_{\ord \mathbf{x} < X} e \left( \sum_{j=1}^{r-2} \beta_{t_j} \sigma_{s, t_j}(\mathbf{x})
+ \theta \sigma_{s, k}(\mathbf{x}) \right).
\notag
\end{eqnarray}
Then for $\mathbf{h} = (h_{t_1}, ..., h_{t_{r-2}}) \in \mathbb{F}_q[t]^{r-2}$, we have
\begin{equation}
\label{minor arc thm eqn 1}
\int_{\mathfrak{m}} \oint |F(\boldsymbol{\beta}, \theta)|^{2s} e \left( \sum_{j=1}^{r-2} - \beta_{t_j} h_{t_j} \right)
\ d \boldsymbol{\beta} \ d \theta = \sum_{\ord \mathbf{x} < X }
\delta(\mathbf{x}, \mathbf{h}) \int_{\mathfrak{m}} e(\theta \sigma_{s,k} (\mathbf{x}) ) \ d \theta,
\end{equation}
where
\begin{equation}
\label{minor arc thm eqn 2}
\delta(\mathbf{x}, \mathbf{h}) = \prod_{j=1}^{r-2} \left( \oint e(\beta_{t_j} (\sigma_{s,t_j}(\mathbf{x}) - h_{t_j} ))
\ d \beta_{t_j} \right).
\end{equation}
Thus, the orthogonality relation ~(\ref{orthogonality relation}) gives us
\begin{eqnarray}
\label{minor arc thm eqn 3}
\oint e(\beta_{t_j} (\sigma_{s,t_j}(\mathbf{x}) - h_{t_j} )) \ d \beta_{t_j} =
\left\{
    \begin{array}{ll}
         1, &\mbox{when } \sigma_{s,t_j}(\mathbf{x}) = h_{t_j}, \\
         0, &\mbox{when } \sigma_{s,t_j}(\mathbf{x}) \not = h_{t_j}. \\
    \end{array}
\right.
\end{eqnarray}
When $\ord \mathbf{x} < X$, we have $\ord  \sigma_{s,t_j}(\mathbf{x}) \leq t_j (X-1)$
for $1 \leq j \leq r-2$, and so it follows from ~(\ref{minor arc thm eqn 2}) and ~(\ref{minor arc thm eqn 3}) that
\begin{equation}
\label{minor arc thm eqn 4}
\sum_{\ord  h_{t_1} \leq t_1 (X-1)} ... \ \sum_{\ord h_{t_{r-2}} \leq t_{r-2} (X-1)} \delta(\mathbf{x}, \mathbf{h}) = 1.
\end{equation}
Since
$$
|g(\theta)|^{2s} = \sum_{ \ord  \mathbf{x} < X}  e(\theta \sigma_{s,k} (\mathbf{x}) ),
$$
we obtain by ~(\ref{minor arc thm eqn 1}) and ~(\ref{minor arc thm eqn 4}),
\begin{eqnarray}
&& \sum_{\ord h_{t_1} \leq t_1 (X-1)} ... \ \sum_{\ord  h_{t_{r-2}} \leq t_{r-2}(X-1)}
\int_{\mathfrak{m}} \oint |F(\boldsymbol{\beta}, \theta)|^{2s} e\left( \sum_{j=1}^{r-2} - \beta_{t_j} h_{t_j} \right)
\ d \boldsymbol{\beta} \ d \theta
\notag
\\
&=&
\int_{\mathfrak{m}} \sum_{ \ord  \mathbf{x} < X}
\left( \sum_{\mathbf{h}} \delta(\mathbf{x}, \mathbf{h})  \right)  e(\theta \sigma_{s,k} (\mathbf{x}) ) \ d \theta
\notag
\\
&=& \int_{\mathfrak{m}} |g(\theta)|^{2s} \ d \theta.
\notag
\end{eqnarray}
It therefore follows by the triangle inequality,
\begin{eqnarray}
\label{T1 eqn 11}
\int_{\mathfrak{m}} |g(\alpha)|^{2s} \ d \alpha &\leq& \sum_{\ord h_{t_1} \leq t_1 (X-1)} ... \ \sum_{\ord  h_{t_{r-2}} \leq t_{r-2}(X-1)}
\int_{\mathfrak{m}} \oint |F(\boldsymbol{\beta}, \theta)|^{2s} \ d \boldsymbol{\beta} \ d \theta
\\
\notag
&\leq&
q^{(\kappa - k - t_{r-1})X} \int_{\mathfrak{m}} \oint |F(\boldsymbol{\beta}, \theta)|^{2s} \ d \boldsymbol{\beta} \ d \theta.
\end{eqnarray}

An argument similar to that employed in the last paragraph permits us to relate the mean value of
$F(\boldsymbol{\beta}, \theta)$ to a sum of integrals involving $f(\boldsymbol{\alpha}, \theta)$
as follows
\begin{equation}
\label{eqn 1}
\int_{\mathfrak{m}} \oint |F(\boldsymbol{\beta}, \theta)|^{2s} \ d \boldsymbol{\beta} \ d \theta
=
\sum_{ \ord  h \leq  t_{r-1} (X-1)}
\int_{\mathfrak{m}} \oint |f(\boldsymbol{\alpha}, \theta)|^{2s} e( - \alpha_{t_{r-1}} h) \ d \boldsymbol{\alpha} \ d \theta.
\end{equation}
The advantage of this maneuver is that we can rewrite the integral in the summand with similar expression involving an extra new
variable $y \in I_X$. We then take the average of these integrals over $y \in I_X$ to get a sharper upper bound for the left hand side of ~(\ref{eqn 1}), which ultimately gives us the desired result. This task will be achieved during the course of the rest of the proof, but first we prove ~(\ref{eqn 1}).
For $h \in \mathbb{F}_q[t]$, let
\begin{equation}
\label{def of delta(x,h)}
\widetilde{\delta}(\mathbf{x},h) =  \oint e(\alpha_{t_{r-1}} (\sigma_{s,t_{r-1}}(\mathbf{x}) - h_{} ))
\ d \alpha_{t_{r-1}} .
\end{equation}
We have by the orthogonality relation ~(\ref{orthogonality relation}),
\begin{eqnarray}
\notag
\widetilde{\delta}(\mathbf{x},h) =
\left\{
    \begin{array}{ll}
         1, &\mbox{when } \sigma_{s,t_{r-1}}(\mathbf{x}) = h_{}, \\
         0, &\mbox{when } \sigma_{s,t_{r-1}}(\mathbf{x}) \not = h_{}. \\
    \end{array}
\right.
\end{eqnarray}
Clearly, $\ord  \mathbf{x} < X$ implies $\ord  \sigma_{s,t_{r-1}}(\mathbf{x}) \leq t_{r-1} (X-1)$. Hence we have
\begin{equation}
\label{minor arc thm eqn 5}
\sum_{\ord  h_{} \leq t_{r-1} (X-1)} \widetilde{\delta}(\mathbf{x}, h) = 1.
\end{equation}
Since $ \overline{f(\boldsymbol{\alpha}, \theta)} = f( - \boldsymbol{\alpha}, - \theta)$, we get
\begin{eqnarray}
|f(\boldsymbol{\alpha}, \theta)|^{2s} &=&
\prod_{i=1}^s \left( \sum_{x_i, x_{s+i} \in I_X } e \left( \sum_{j=1}^{r-1} \alpha_{t_j} (x_i^{t_j} - x_{s+i}^{t_j} )
+ \theta (x_i^k - x_{s+i}^k) \right) \right)
\notag
\\
&=& \sum_{\ord \mathbf{x} < X} e \left( \sum_{j=1}^{r-1} \alpha_{t_j} \sigma_{s, t_j}(\mathbf{x})
+ \theta \sigma_{s, k}(\mathbf{x}) \right).
\notag
\end{eqnarray}
Thus, it follows by ~(\ref{def of delta(x,h)}) that
\begin{equation}
\label{minor arc thm eqn 6}
\int_{\mathfrak{m}} \oint |f(\boldsymbol{\alpha}, \theta)|^{2s} e \left( - \alpha_{t_{r-1}} h_{} \right)
\ d \boldsymbol{\alpha} \ d \theta = \sum_{\ord  \mathbf{x} < X }
\widetilde{\delta}(\mathbf{x}, h) \int_{\mathfrak{m}} \oint e\left( \sum_{j=1}^{r-2}  \beta_{t_j} \sigma_{s, t_j}(\mathbf{x}) + \theta \sigma_{s,k}(\mathbf{x})  \right) \ d \boldsymbol{\beta} \ d \theta.
\end{equation}

Therefore, we obtain by ~(\ref{minor arc thm eqn 5}) and  ~(\ref{minor arc thm eqn 6}),
\begin{eqnarray}
&&\sum_{\ord  h_{} \leq t_{r-1} (X-1)} \int_{\mathfrak{m}} \oint |f(\boldsymbol{\alpha}, \theta)|^{2s} e \left( - \alpha_{t_{r-1}} h_{} \right)
\ d \boldsymbol{\alpha} \ d \theta
\\
\notag
&=& \sum_{\ord  \mathbf{x} < X } \
\sum_{\ord  h_{} \leq t_{r-1} (X-1)} \widetilde{\delta}(\mathbf{x}, h) \int_{\mathfrak{m}} \oint e\left( \sum_{j=1}^{r-2} \beta_{t_j} \sigma_{s, t_j}(\mathbf{x}) + \theta \sigma_{s,k}(\mathbf{x})  \right) \ d \boldsymbol{\beta} \ d \theta
\\
\notag
&=&
\int_{\mathfrak{m}} \oint \sum_{\ord  \mathbf{x} < X } e\left( \sum_{j=1}^{r-2} \beta_{t_j} \sigma_{s, t_j}(\mathbf{x}) + \theta \sigma_{s,k}(\mathbf{x})  \right) \ d \boldsymbol{\beta} \ d \theta
\\
\notag
&=&
\int_{\mathfrak{m}} \oint |F(\boldsymbol{\beta}, \theta)|^{2s} \ d \boldsymbol{\beta} \ d \theta,
\end{eqnarray}
which is exactly the equation ~(\ref{eqn 1}) we aimed to prove.

Given $y \in I_X$, observe that $I_X$ is invariant under translation by $y$, or in other words
$$
I_X = \{ x : x \in \mathbb{F}_q[t], \ord x < X  \} = \{ x + y :x \in \mathbb{F}_q[t], \ord x < X  \}.
$$
Let
$$
\lambda(z; \boldsymbol{\alpha} ) = \sum_{j=1}^{r-1} \alpha_{t_j} z^{t_j} + \alpha_k z^k.
$$
By the above observation, shifting the variable of summation in $f(\boldsymbol{\alpha})$ by $y$ gives us
\begin{equation}
\label{shifted exp sum}
f(\boldsymbol{\alpha}) = 
\sum_{x \in I_X} e \left( \lambda(x; \boldsymbol{\alpha} ) \right)
=\sum_{x \in I_X} e \left( \lambda(x - y; \boldsymbol{\alpha} ) \right).
\end{equation}

Define $\Delta (\theta, h, y)$ as follows:
$$
\Delta (\theta, h, y) =  e(\theta \sigma_{s,k}(\mathbf{x} -y)),
$$
when the $2s$-tuple $\mathbf{x}$ satisfies
\begin{equation}
\label{eqn 3}
\sum_{i=1}^s ( (x_i - y)^{t_j} - (x_{s+i} - y)^{t_j} ) = 0 \phantom{12} (1 \leq j \leq r-2)
\end{equation}
and
\begin{equation}
\label{eqn 3'}
\sum_{i=1}^s ( (x_i - y)^{t_{r-1}} - (x_{s+i} - y)^{t_{r-1}} ) = h.
\end{equation}
Otherwise, we let $\Delta (\theta, h, y) = 0$.
Substituting the expression ~(\ref{shifted exp sum}) for $f(\boldsymbol{\alpha}, \theta)$,
we find by the orthogonality relation ~(\ref{orthogonality relation}),
\begin{equation}
\label{eqn 4}
\oint
|f(\boldsymbol{\alpha}, \theta)|^{2s} e(- \alpha_{t_{r-1}} h)
 \ d \boldsymbol{\alpha}
= \sum_{\ord  \mathbf{x} < X} \Delta (\theta, h, y).
\end{equation}
We now simplify the function $\Delta (\theta, h, y)$ and obtain another expression for
the left hand side of ~(\ref{eqn 4}).
First, we prove that the $2s$-tuple $\mathbf{x}$ satisfies ~(\ref{eqn 3}) and ~(\ref{eqn 3'})
if and only if $\mathbf{x}$ satisfies
\begin{equation}
\label{minor arc thm eqn 8}
\sum_{i=1}^s ( x_i^{t_j} - x_{s+i}^{t_j} ) = 0 \phantom{12} (1 \leq j \leq r-2)
\end{equation}
and
\begin{equation}
\label{minor arc thm eqn 7}
\sum_{i=1}^s ( x_i^{t_{r-1}} - x_{s+i}^{t_{r-1}} ) = h.
\end{equation}

Suppose $\mathbf{x}$ satisfies ~(\ref{eqn 3}) and ~(\ref{eqn 3'}).
Since $\mathbb{F}_q$ has characteristic $p$, we have $(x-y)^p = x^p - y^p$.
Recall $t_{r-1} = j_0$. Thus, we can prove by induction and the definition of $\mathcal{R}'$ that 
~(\ref{eqn 3}) implies
\begin{equation}
\label{eqn 3''}
\sum_{i=1}^s ( (x_i - y)^{j} - (x_{s+i} - y)^{j} ) = 0 \phantom{12} (1 \leq j < t_{r-1}).
\end{equation}
Note we can verify that $t_{r-1}>1$ for the cases we consider here.
By applying the binomial theorem, we obtain that whenever
a $2s$-tuple $\mathbf{x}$ satisfies
~(\ref{eqn 3'}) and the system ~(\ref{eqn 3''}), then $\mathbf{x}$ satisfies
\begin{equation}
\label{eqn 5}
\sum_{i=1}^s ( x_i^j - x_{s+i}^j ) = 0 \phantom{12} (1 \leq j < t_{r-1})
\end{equation}
and ~(\ref{minor arc thm eqn 7}).
Clearly the system ~(\ref{eqn 5}) implies ~(\ref{minor arc thm eqn 8}).
For the converse direction, since ~(\ref{minor arc thm eqn 8}) implies ~(\ref{eqn 5}), we can obtain the desired
result in a similar manner as in the forward direction.

Suppose $\mathbf{x}$ satisfies ~(\ref{minor arc thm eqn 8}) and ~(\ref{minor arc thm eqn 7}),
and consequently ~(\ref{eqn 5}).
If $p \nmid (k-1)$, then $t_{r-1} = j_0 = k-1$ and we have
\begin{equation}
\label{minor arc thm eqn 10}
\sigma_{s,k}(\mathbf{x} -y) = \sum_{i=1}^s ( (x_i - y)^{k} - (x_{s+i} - y)^{k} )
= \sigma_{s,k}(\mathbf{x}) - chy^{k - t_{r-1}},
\end{equation}
where $c = { k \choose t_{r-1} } \not \equiv 0 \ (\text{mod } p)$.
If $k = m p^b + 1$, then we can deduce from $t_{r-1} = j_0 = (m-1)p^b + 1 > m$, which we note does not hold if $m=1$,  
and ~(\ref{eqn 5}) that
$$
\sum_{i=1}^s x_i^{k-1} - x_{s+i}^{k-1} = \left( \sum_{i=1}^s x_i^{m} -x_{s+i}^{m} \right)^{p^b} = 0.
$$
Therefore, by the binomial theorem, the above equation, and the definition of $j_0$
given in ~(\ref{def j_0}), we also obtain ~(\ref{minor arc thm eqn 10}) when $k$ is of the form $k = m p^b + 1$, $m>1$.
Thus, we can rewrite the definition of $\Delta (\theta, h, y)$ as
$$
\Delta (\theta, h, y) = e( \theta \sigma_{s,k}(\mathbf{x}) - c h y^{k - t_{r-1}} \theta ),
$$
whenever $\mathbf{x}$ satisfies ~(\ref{minor arc thm eqn 8}) and ~(\ref{minor arc thm eqn 7});
otherwise, $\Delta (\theta, h, y)$ is equal to $0$.
Thus, we have
$$
\oint
|f(\boldsymbol{\alpha}, \theta)|^{2s} e(-chy^{k - t_{r-1}} \theta - \alpha_{t_{r-1}} h)
\ d \boldsymbol{\alpha}
=
\sum_{\ord{\mathbf{x}} < X } \Delta (\theta, h, y),
$$
and consequently, it follows from ~(\ref{eqn 4}) that
$$
\oint
|f(\boldsymbol{\alpha}, \theta)|^{2s} e(- \alpha_{t_{r-1}} h) \ d \boldsymbol{\alpha}
=
\oint
|f(\boldsymbol{\alpha}, \theta)|^{2s} e(-chy^{k - t_{r-1}} \theta - \alpha_{t_{r-1}} h) \ d \boldsymbol{\alpha}.
$$
From here, we have by ~(\ref{eqn 1}),
\begin{eqnarray}
\notag
&&\int_{\mathfrak{m}} \oint |F(\boldsymbol{\beta}, \theta)|^{2s} \ d \boldsymbol{\beta} \ d \theta
\\
&=&
\label{eqn 6}
\int_{\mathfrak{m}} \oint |f(\boldsymbol{\alpha}, \theta)|^{2s}
\sum_{ \ord  h \leq t_{r-1} (X-1)} e(-chy^{k - t_{r-1}}\theta - \alpha_{t_{r-1}} h) \ d \boldsymbol{\alpha} \ d \theta.
\end{eqnarray}
Since the left hand side of ~(\ref{eqn 6}) is independent of $y$, we can average the right hand side over $y \in I_X$
to obtain
\begin{eqnarray}
\notag
&&\int_{\mathfrak{m}} \oint |F(\boldsymbol{\beta}, \theta)|^{2s} \ d \boldsymbol{\beta} \ d \theta
\\
\notag
&=&
q^{-X} \sum_{y \in I_X}
\int_{\mathfrak{m}} \oint |f(\boldsymbol{\alpha}, \theta)|^{2s}
\sum_{ \ord h \leq t_{r-1} (X-1)} e(-chy^{k - t_{r-1}} \theta - \alpha_{t_{r-1}} h) \ d \boldsymbol{\alpha} \ d \theta
\\
&=&
\label{minor arc thm eqn 9}
\int_{\mathfrak{m}} \oint |f(\boldsymbol{\alpha}, \theta)|^{2s} \psi(\theta, \alpha_{t_{r-1}}) \ d \boldsymbol{\alpha} \ d \theta.
\end{eqnarray}
In the last equality displayed above, we invoked ~(\ref{def psi}), the definition of $\psi(\theta, \alpha)$.
We apply the appropriate lemma depending on $k$ from Section \ref{technical section}, namely Lemmas \ref{lemma case 1},
\ref{lem 7} and \ref{lem 8},
to $\psi(\theta, \alpha)$
and obtain an upper bound for the right hand side of ~(\ref{minor arc thm eqn 9}).
We then use the resulting estimate and ~(\ref{def of J}) to bound ~(\ref{T1 eqn 11}), from which we obtain
\begin{eqnarray}
\int_{\mathfrak{m}} |g(\alpha)|^{2s} \ d \alpha
\ll
q^{(\kappa - k - t_{r-1})X} q^{(j_0 - \delta)X} J_s(\mathcal{R}', X) = q^{(\kappa - k - \delta)X} J_s(\mathcal{R}',X),
\end{eqnarray}
for suitable $\delta>0$. 
\end{proof}

\subsection{Case $k = p^b+1$}
\label{minor arc case 2}

Recall from above that if $k = p^b + 1$, then $r = (1 - 1/p)(k-p^b) + (1 + 1/p) = 2.$
We obtain the following minor arc bound when $k = p^b + 1$.
\begin{thm}
\label{Thm minor arc bound 1'}
Suppose $k \geq 3$ and $k = p^b + 1$. Let $\kappa = 1 + k$, $\mathcal{R}' = \{ 1, k\}$, and 
\begin{equation}
\notag
\delta_0 = \frac{1}{16(p^b + 2)}.
\end{equation}
Then we have
$$
\int_{\mathfrak{m}} |g(\alpha)|^{2s+ 1} \ d \alpha
\ll
q^{(1 + \kappa - k - \delta_0)X} J_s(\mathcal{R}',X),
$$
where the implicit constant depends only on $q$ and $k$.
\end{thm}

By applying Theorem ~\ref{eff cong result} to Theorem \ref{Thm minor arc bound 1'},
we also obtain the following corollary.
\begin{cor}
\label{minor arc bound 2'}
Suppose $k \geq 3$, $k = p^b + 1$, and $s \geq (2k+2)$. Let
$$
\delta_0 = \frac{1}{16(p^b + 2)}.
$$
Then for each $\epsilon > 0$, we have
$$
\int_{\mathfrak{m}} |g(\alpha)|^{2s+ 1} \ d \alpha
\ll
q^{(2s + 1 - k - \delta_0 + \epsilon)X},
$$
where the implicit constant depends only on $s, q, k$, and $\epsilon$.
\end{cor}

We introduce some notation before we get into the proof of Theorem \ref{Thm minor arc bound 1'}.
Given $j,j' \in \mathbb{Z}^+$, we write $j \preceq_p j'$ if $p \nmid {j' \choose j}$.
By Lucas' Theorem, this happens precisely when all the digits of $j$ in base $p$
are less than or equal to the corresponding digits of $r$. From this characterization,
it is easy to see that the relation $\preceq_p$ defines a partial order on $\mathbb{Z}^+$.
If $j \preceq_p j'$, then we necessarily have $j \leq j'$. Let $\mathcal{K} \subseteq \mathbb{Z}^+$.
We say an element $k \in \mathcal{K}$  is \textit{maximal} if it is maximal with respect to
$\preceq_p$, that is, for any $j \in \mathcal{K}$, either $j \preceq_p k$
or $j$ and $k$ are not comparable. Following the notation of \cite{LL}, we define
the \textit{shadow} of $\mathcal{K}$, $\mathcal{S(K)}$, to be
$$
\mathcal{S(K)} = \left\{  j \in \mathbb{Z}^+ : 
j \preceq_p j' \text{ for some } j' \in \mathcal{K}  \right\}.
$$
We also define
$$
\mathcal{K}^* = \left\{ k \in \mathcal{K} : p \nmid k \text{ and }
p^vk \not \in \mathcal{S(K)} \text{ for any } v \in \mathbb{Z}^+   \right\}.
$$

We invoke the following result from \cite{LL}. The theorem allows us to estimate
certain coefficients of a polynomial $h(u)$ by an element in $\mathbb{K}$ when the exponential sum
of $h(u)$ is sufficiently large. We use the result to bound exponential sums over the minor arcs.

\begin{thm}[Theorem 12, \cite{LL}]   
\label{thm LL 1}
Let $\mathcal{K} \subseteq \mathbb{Z}^+$ and $h(u) = \sum_{j \in \mathcal{K} \cup \{ 0 \}} \alpha_j u^j \in \mathbb{K}_{\infty}[u]$, 
where $\alpha_j \not = 0 \ (j \in \mathcal{K})$.
Suppose that $k \in \mathcal{K}^*$ is maximal in $\mathcal{K}$. Then there exist
constants $c, C > 0$, depending only on $\mathcal{K}$ and $q$, such that the following holds:
suppose that for some $0 < \eta \leq c X$, we have
$$
\Big{|} \sum_{x \in I_X}  e(h(x)) \Big{|} \geq q^{X- \eta}.
$$
Then for any $\epsilon > 0$ and $X$ sufficiently large in terms of $\mathcal{K}$, $\epsilon$ and $q$,
there exist $a, g \in \mathbb{F}_q[t]$ such that
$$
\ord (g \alpha_k - a) < -k X + \epsilon X + C \eta \ \ \text{  and  } \ \ \ord g \leq \epsilon X + C \eta.
$$
\end{thm}

\begin{proof}[Proof of Theorem \ref{Thm minor arc bound 1'}]
We  bound $\sup_{\theta \in \mathfrak{m}} | g(\theta) |$ using Theorem \ref{thm LL 1}.
For $g(\theta)$ with $k = p^b + 1$, we have $\mathcal{K} = \{ k \}$, and
thus
$$
\mathcal{S(K)} = \{ k, p^b, 1 \},
$$
and
$$
\mathcal{K}^* = \{ k \}.
$$
Clearly, $k$ is maximal in $\mathcal{K}$.
We also have
\begin{eqnarray}
\mathcal{S(K)}' &:=& \{ i \in \mathbb{N} : p \nmid i \text{ and } p^v i \in \mathcal{S(K)} \text{ for some } v \in \mathbb{N} \cup \{0\} \}
\notag
\\
&=& \{ k, 1 \}.
\notag
\end{eqnarray}

It is given at the end of the proof of \cite[Theorem 12]{LL} that we may take $c = 1/(8(r_0 \phi + r_0))$ and $C = 2 (r_0 \phi + r_0)$, where
$r_0 = \# \ \mathcal{S(K)}'$ and $\phi = \max_{i \in \mathcal{S(K)}' } i$.
Therefore, we can apply Theorem \ref{thm LL 1} with
$$
c = \frac{1}{8(2k+2)} = \frac{1}{16(p^b+2)} \ \ \text{  and  } \ \ C = 2 (2k+2).
$$
Take any $\theta \in \mathfrak{m}$. We set $\epsilon = 1/2$. Suppose for some $X$ sufficiently large,
with respect to $\mathcal{K}$ and $q$, we have
$$
|g(\theta)| \geq q^{X - c X}.
$$
Then, by Theorem \ref{thm LL 1}, there exist $\tilde{g}, \tilde{a} \in \mathbb{F}_q[t]$ such that
$$
\ord (\tilde{g} \theta - \tilde{a}) < -k X + \epsilon X + \frac{1}{4} X \ \ \text{  and  } \ \ \ord \tilde{g} \leq \epsilon X + \frac{1}{4} X.
$$
Let $(\tilde{g}, \tilde{a}) = \ell$, and denote $\tilde{g} = \ell g_0 $ and $\tilde{a} = \ell a_0$.
We obtain from above inequalities,
$$
\ord (\theta - a_0/g_0) = \ord (\theta - \tilde{a}/\tilde{g}) < -k X + \epsilon X + \frac{1}{4} X - \ord \tilde{g} \leq - (k-1) X - \ord {g_0}
$$
and
$$
\ord g_0 \leq \ord \tilde{g} \leq \epsilon X + \frac{1}{4} X < X.
$$
By the definition of major arcs ~(\ref{def major}), this implies that $\theta \in \mathfrak{M}_k$, which is a contradiction.
Therefore, we must have
$$
|g(\theta)| < q^{X - c X}
$$
for all $X$ sufficiently large with respect to $\mathcal{K}$ and $q$.
Since the result is independent of the choice of $\theta \in \mathfrak{m}$,
it follows that
\begin{equation}
\label{sup over minor}
\sup_{\theta \in \mathfrak{m}_k} |g (\theta)| < q^{X - \frac{1}{16(p^b+2)} X }.
\end{equation}

When $k = p^b + 1$, we have $r=2$; therefore, we have $F( \boldsymbol{ \beta}, \theta) = g(\theta)$.
Thus, we obtain by ~(\ref{eqn 1}) and the triangle inequality that
\begin{eqnarray}
&& \int_{\mathfrak{m}} |g(\theta)|^{2s+1} \ d \theta
\label{minor arc estimate for m=1}
\\
&\leq&
\sup_{\theta \in \mathfrak{m}} | g(\theta) | \cdot \int_{\mathfrak{m}} |g(\theta)|^{2s} \ d \theta
\notag
\\
&=&
\sup_{\theta \in \mathfrak{m}} | g(\theta) |
\cdot \sum_{h \in I_X}
\int_{\mathfrak{m}} \oint |f(\alpha, \theta)|^{2s} e\left( -\alpha h \right)
\ d \alpha \ d \theta
\notag
\\
&\leq&
\sup_{\theta \in \mathfrak{m}} | g(\theta) |
\cdot q^X \cdot \oint \oint |f(\alpha, \theta)|^{2s}
\ d \alpha \ d \theta
\notag
\\
&=&
\sup_{\theta \in \mathfrak{m}} | g(\theta) |
\cdot q^X \cdot
J_{s}(\mathcal{R}', X).
\notag
\end{eqnarray}
Consequently, substituting ~(\ref{sup over minor}) into the above inequality ~(\ref{minor arc estimate for m=1}) gives us
$$
\int_{\mathfrak{m}} |g(\theta)|^{2s+1} \ d \theta \ll q^{2 X - \frac{1}{16(p^b+2)} X } J_{s}(\mathcal{R}', X).
$$
\end{proof}

\section{Weyl Differencing}
\label{section weyl}
Let $w_0(u)$ be a polynomial in $\mathbb{F}_q[t][u]$. 
Let $z_1, ..., z_h$ be indeterminates. We define the differencing operator $\Delta_{z_1}$ by
$$
\Delta_{z_1} (w_0)(u) = w_0(u + z) - w_0(u) \in \mathbb{F}_q[t][u, z_1],
$$
where we denote $\Delta_{z_1} (w_0) = \Delta_{z_1} (w_0)(u)$. We also define recursively
$$
\Delta_{z_h} ... \Delta_{z_1}(w_0) (u)= \Delta_{z_{h-1}} ... \Delta_{z_1}(w_0) (u + z_h) - \Delta_{z_{h-1}} ... \Delta_{z_1}(w_0) (u)
\in \mathbb{F}_q[t][u, z_1, ..., z_h],
$$
and we denote $\Delta_{z_h} ... \Delta_{z_1}(w_0) = \Delta_{z_h} ... \Delta_{z_1}(w_0)(u).$

While in characteristic zero the above differencing process, known as Weyl differencing,
decreases the degree (in $u$) of the polynomial by one, the situation in positive
characteristic is more subtle. With application of Hua's lemma (Proposition \ref{Hua's lemma}) in mind, it will be useful to know
how many times one can apply Weyl differencing to $u^k$ in $\mathbb{F}_q[t][u]$ before it becomes identically zero.
Note that given an indeterminate $z$ and a monomial $u^{\ell}$, we have $\Delta_{z} (u^{\ell}) = 0$ if and only if ${\ell}=0$. To see this, suppose we have
${\ell} \geq 1$ and
$$
0 = \Delta_{z} (u^{\ell}) = (u + z)^{\ell} - u^{\ell} = \sum_{j=0}^{{\ell}-1} {{\ell} \choose j} u^j z^{{\ell}-j}.
$$
Then, in particular it must be that ${{\ell} \choose 0} = 1 \equiv 0 \ (\text{mod } p)$,
which is a contradiction. Therefore, we have ${\ell}=0$. The converse direction is trivial.
The following lemma is a slight modification of
\cite[Lemma 8.1]{LW} and we omit the proof here.

\begin{lem}
\label{counting weyl}
Let $k = c_v p^v + ... + c_0$ with $0 \leq c_i < p \ (0 \leq i \leq v)$, and
let $h_0 = h_0(k) = c_v + ... + c_0$. Let $z_1, ..., z_{h_0 + 1}$ be indeterminates.
Then, we have
$$
0 \not = \Delta_{z_{h_0} } ...  \Delta_{z_1} u^k \in \mathbb{F}_q[t] [u, z_1, ..., z_{h_0}]
$$
and
$$
0 = \Delta_{z_{h_0 + 1} } ...  \Delta_{z_1} u^k  \in \mathbb{F}_q[t] [u, z_1, ..., z_{h_0 + 1}].
$$
\end{lem}

Combining Lemma \ref{counting weyl} and \cite[Proposition 13]{RMK}, we have the following
version of Hua's lemma.
\begin{prop}
\label{Hua's lemma}
Let $w_0(u)$ be a polynomial in $\mathbb{F}_q[t][u]$ of degree $k$ in $u$,
and let $w(\alpha) = \sum_{x \in I_X} e( w_0(x)  \alpha)$.
Let $h_0(k)$ be as defined in the statement of Lemma \ref{counting weyl}.
Suppose $j \leq h_0(k)$.
Then for every $\epsilon > 0$, we have
$$
\oint | w(\alpha) |^{2^j} \ d \alpha \ll q^{(2^j - j + \epsilon)X},
$$
where the implicit constant depends only on $k, q$, and $\epsilon$.
\end{prop}
We apply Proposition \ref{Hua's lemma} in Sections \ref{asymptotic section} and \ref{exceptional section}
with $w_0(u) = u^k$.

\section{Asymptotic Formula and $\tilde{G}_q(k)$}
\label{asymptotic section}

We now lower the bound on $s$ in Corollary \ref{minor arc bound 2}
via combination of Proposition \ref{Hua's lemma} and H\"{o}lder's inequality,
and obtain Theorems \ref{asymptotic stuff} and \ref{asymptotic stuff 2}.
First, we consider the case when $p \nmid (k-1)$ in Proposition \ref{prop s_1 first}.
We then take care of the case $k = m p^b + 1$ in Proposition \ref{prop s_1 mpb}.

Let
$$
s'_0(j) = 2k^2 + 1 - \Bigg{\lceil} \frac{2kj-2^j}{k+1-j} \Bigg{\rceil}.
$$

If $k < p$, we set
\begin{equation}
\label{def s_1}
s_1(k) = \min_{\stackrel{1 \leq j < k}{ 2^j \leq k(2k+1) }} s'_0(j).
\end{equation}
On the other hand, if $k > p$ and $p \nmid (k-1)$, we set
\begin{equation}
\label{def s_1 second}
s_1(k) = 2rk  + 1 - \Bigg{\lceil}  \frac{ 6r-8}{k-2} \Bigg{\rceil}.
\end{equation}

\begin{prop}
\label{prop s_1 first}
Suppose $k \geq 3$, $p \nmid k$, and 
$p \nmid (k-1)$.
Let $s_1(k)$ be as given in ~(\ref{def s_1}) when $k < p$ and in ~(\ref{def s_1 second})
when  $k > p$. 
If $s \geq s_1(k)$, then there exists $\delta_1 > 0$  such that
$$
\int_{\mathfrak{m}} |g(\alpha)|^{s} \ d \alpha
\ll q^{ ( s - k  -  \delta_1 ) X},
$$
where the implicit constant depends only on $s, q, k, \mathcal{R}'$, and $\delta_1$.
\end{prop}

\begin{proof}
Let $h_0(k)$ be as in the statement of Lemma \ref{counting weyl}.
We have by Proposition \ref{Hua's lemma}, if $j \leq h_0(k)$, then for any $\epsilon > 0$,
\begin{equation}
\label{Hua}
\oint |g(\alpha)|^j \ d \alpha  \ll q^{(2^j - j + \epsilon)X}.
\end{equation}
We let $s_0(j) = 2 r (k+1) a' + 2^j b'$, where $a'+b' = 1$. Then H\"{o}lder's inequality
gives us
\begin{equation}
\label{ineq holder s}
\int_{\mathfrak{m}} |g(\alpha)|^{s_0(j)} \ d \alpha
\leq \left( \int_{\mathfrak{m}} |g(\alpha)|^{2r(k+1)} \ d \alpha  \right)^{a'}
\left( \oint |g(\alpha)|^{2^j} \ d \alpha \right)^{b'}.
\end{equation}

Recall for the range of $k$ we are considering, we can take $\delta_0 = 1$ in Corollary \ref{minor arc bound 2}.
We consider $j$ in the following range: $1 \leq j < k$, $2^j \leq (2r-1)(k+1)+1$ and $j \leq h_0(k)$.
Define
$$
\eta(j) = \frac{2rj}{k-j + 1} - \frac{2^j }{k-j + 1}
$$
and
let
$$
\gamma(j) = 1 + \eta(j) - \lceil \eta(j) \rceil.
$$
We choose
$$
a' = \frac{k-j}{k-j + 1} + \frac{ \gamma(j) } {2r(k+1)-2^j }
$$
and
$$
b' = \frac{1}{k-j + 1} - \frac{ \gamma(j) } {2r(k+1)-2^j }.
$$
Note that our restriction on $j$ ensures $b' > 0.$
Also, this choice of $a'$ and $b'$ ensures
$a' - (k-j)b' > 0$. Then, by Corollary \ref{minor arc bound 2}
and ~(\ref{Hua}), we have the following bound for ~(\ref{ineq holder s}):
$$
\int_{\mathfrak{m}} |g(\alpha)|^{s_0(j)} \ d \alpha
\ll q^{\epsilon X} q^{a'(2r(k+1) - k - 1)X} q^{b'(2^j - j)X}
\ll q^{ ( s_0(j) - k  -  ( a' - (k-j)b' ) + \epsilon ) X}.
$$
By the trivial bound $|g(\alpha)| \leq q^X$, it follows that for any $s \geq s_0(j)$ we have
$$
\int_{\mathfrak{m}} |g(\alpha)|^{s} \ d \alpha
\ll q^{(s - s_0(j))X} \int_{\mathfrak{m}} |g(\alpha)|^{s_0(j)} \ d \alpha
\ll q^{ ( s - k  -  ( a' - (k-j)b' ) + \epsilon ) X}.
$$

We can simplify $s_0(j)$ as
\begin{eqnarray}
\label{bound on s0(j)}
s_0(j) &=&  2 r (k+1) \left( \frac{k-j}{k-j + 1} + \frac{ \gamma(j) } {2r(k+1)-2^j } \right)
+ 2^j \left( \frac{1}{k-j + 1} - \frac{ \gamma(j) } {2r(k+1)-2^j } \right)
\notag
\\
&=& 2rk
- \eta(j) + \gamma(j)
\notag
\\
&=& 2rk
+ 1 - \lceil \eta(j) \rceil.
\notag
\end{eqnarray}
To establish our result, all we have left is to choose $j$ within the appropriate range
given above such that $s_0(j)$ is as small as possible. This value of $s_0(j)$ will be our $s_1(k)$.
We consider the two cases separately.

Case 1: $k > p$. 
From $p \nmid k$, $p \nmid (k-1)$, and $k > p$, we can verify that $3 \leq h_0(k)$.
Thus we know we can apply Weyl differencing at least three times. Therefore,
we set $s_1(k) = s_0(3)$.
Since
\begin{equation}
\label{eta 2 first}
0 <  \eta(3) = \frac{ 6r-8}{k-2},
\end{equation}
we obtain
$$
s_1(k) =  2rk  + 1 - \Bigg{\lceil}  \frac{ 6r - 8 }{k-2} \Bigg{\rceil}.
$$

Case 2: $k < p$. In this case, we have $h_0(k) = k$.
We set
\begin{equation}
\label{asympt eqn 1}
s_1(k) = \min_{\stackrel{1 \leq j < k}{ 2^j \leq (2r-1)(k+1)+1 }} s_0(j).
\end{equation}
Since $r =  k - \lfloor k/p \rfloor = k$, we have $s_0(j) = s'_0(j)$ and $(2r-1)(k+1)+1 = k(2k+1).$
Therefore, we see that $s_1(k)$ given above in ~(\ref{asympt eqn 1}) coincides with ~(\ref{def s_1}).

\end{proof}

Now we consider the case $k = mp^b + 1$.
If $m=1$, we set
$s_1(k) = 4k+5.$
If $m>1$, then we set
\begin{equation}
\label{asympt eqn 3}
s_1(k) =  2rk + 2r - \Bigg{\lfloor} \frac{(m-1)(1- 1/p)}{2} \Bigg{\rfloor}.
\end{equation}

\begin{prop}
\label{prop s_1 mpb}
Suppose $k = m p^b + 1$ with $p \nmid m$. Let $s_1(k)$ be $4k+ 5$ 
when $m=1$ and as in ~(\ref{asympt eqn 3}) when $m>1$.
If $s \geq s_1(k)$, then there exists $\delta_1 > 0$  such that
$$
\int_{\mathfrak{m}} |g(\alpha)|^{s} \ d \alpha
\ll q^{ ( s - k  -  \delta_1 ) X},
$$
where the implicit constant depends only on $s, q, k, \mathcal{R}'$, and $\delta_1$.
\end{prop}

\begin{proof}
We first deal with the case $m>1$.
Let $h_0(k)$ be as in the statement of Lemma \ref{counting weyl}.
If $j \leq h_0(k)$, then for any $\epsilon > 0$ we have
~(\ref{Hua}).
We let $s_0(j) = 2 r (k+1) a' + 2^j b'$, where $a'+b' = 1$, as before in Proposition \ref{prop s_1 first}.
Then by H\"{o}lder's inequality, we have ~(\ref{ineq holder s}).
We consider $j$ in the following range: $1 \leq j < k$, $2^j \leq (2r-1)(k+1)+1$ and $j \leq h_0(k)$.
Let $\epsilon(j)$ be a small positive number.
We choose
$$
a' = \frac{k-j}{k-j + \delta} + \frac{ \epsilon(j) } {2r(k+1)-2^j }
$$
and
$$
b' = \frac{\delta}{k-j + \delta} - \frac{ \epsilon(j) } {2r(k+1)-2^j },
$$
where we let $\delta = \delta_0 = 1/(4p^b)$ from Corollary \ref{minor arc bound 2}.

Note that we pick $\epsilon(j)$ sufficiently small to make sure $b' > 0$.
Also, the range of $j$ we are considering and this choice of $a'$ and $b'$ ensure
$$
\delta a' - (k-j)b' = \frac{ (\delta + k -j) \epsilon(j) } {2r(k+1)-2^j } > 0.
$$ By Corollary \ref{minor arc bound 2}
and ~(\ref{Hua}), we have the following bound for ~(\ref{ineq holder s}):
$$
\int_{\mathfrak{m}} |g(\alpha)|^{s_0(j)} \ d \alpha
\ll q^{\epsilon X} q^{a'(2r(k+1) - k - \delta)X} q^{b'(2^j - j)X}
\ll q^{ ( s_0(j) - k  -  ( \delta a' - (k-j)b' ) + \epsilon ) X}.
$$
By the trivial bound $|g(\alpha)| \leq q^X$, it follows that for any $s \geq s_0(j)$
we have
$$
\int_{\mathfrak{m}} |g(\alpha)|^{s} \ d \alpha
\ll q^{(s - s_0(j))X} \int_{\mathfrak{m}} |g(\alpha)|^{s_0(j)} \ d \alpha
\ll q^{ ( s - k  -  ( \delta a' - (k-j)b' ) + \epsilon ) X}.
$$

We can simplify $s_0(j)$ as
\begin{eqnarray}
\label{bound on s0(j) mpb}
s_0(j) &=&  2 r (k+1) \left( \frac{k-j}{k-j + \delta} + \frac{ \epsilon(j) } {2r(k+1)-2^j } \right)
+ 2^j \left( \frac{\delta}{k-j + \delta} - \frac{ \epsilon(j) } {2r(k+1)-2^j } \right)
\notag
\\
&=& 2rk + 2(1-\delta)k - \frac{2r(j + (1-\delta) (\delta-j))}{k-j + \delta}
+ \frac{2^j \delta}{k-j + \delta} + \epsilon(j)
\notag
\\
&=& 2rk
+ 2(1-\delta)r
- \delta \frac{2 r (1 + j- \delta) - 2^j}{k-j+\delta} + \epsilon(j)
\notag
\\
&=& 2rk + 2r
- \delta \frac{2 r (k+1) - 2^j }{k-j+\delta} + \epsilon(j).
\notag
\end{eqnarray}
To establish our result, all we have left is to choose $j$ within the appropriate range
given above such that $s_0(j)$ is as small as possible.
We would like to maximize the value
$$
\delta \frac{2 r (k+1) - 2^j }{k-j+\delta}
$$
in order to minimize $s_0(j)$. We then let the smallest integer greater than the $s_0(j)$ found to
be our $s_1(k)$.

Since $m>1$, we can verify that $h_0(k) \geq 3$. Thus
we know we can apply Weyl differencing at least three times. 
We have
$$
r = (1 - 1/p)(k-p^b) + (1 + 1/p) = (m-1)(p^b-p^{b-1})+2.
$$
Also, recall from above we have set $\delta = \delta_0 =  1/ (4 p^b)$.
Let $j=3$ and we obtain
\begin{eqnarray}
s_0(3) &=&
2rk + 2r - \frac{2 r (k+1) - 2^3 }{4 p^b(k-3+\delta)} + \epsilon(3)
\notag
\\
&=&
2rk + 2r - \frac{2 (m-1)(p^b-p^{b-1}) (k+1)}{4 p^b(k-3+\delta)} - \frac{4(k+1)}{4 p^b(k-3+\delta)} + \frac{8}{4 p^b(k-3+\delta)} + \epsilon(3)
\notag
\\
&=&
2rk + 2r - \frac{(m-1)(1 - 1/p) (k+1)}{2(k-3+\delta)} - \frac{ k - 1 }{p^b(k-3+\delta)} + \epsilon(3)
\notag
\\
&\leq&
2rk + 2r - \frac{(m-1)(1- 1/p)(k+1)}{2(k-3+\delta)}
\notag
\\
&\leq&
2rk + 2r - \frac{(m-1)(1- 1/p)}{2}.
\notag
\end{eqnarray}
Therefore, we let $s_1(k) =  \Big{\lceil} 2rk + 2r - \frac{(m-1)(1- 1/p)}{2} \Big{\rceil}
= 2rk + 2r - \Big{\lfloor} \frac{(m-1)(1- 1/p)}{2} \Big{\rfloor}
\geq s_0(3).$

The case $m=1$ is an immediate consequence of Corollary \ref{minor arc bound 2'}. When $m=1$, we have
$r=2$ and the saving in the exponent of $\delta_0 = \frac{1}{16(p^b + 2)}$ from Corollary \ref{minor arc bound 2'}, however, with these values our approach above is not effective as in the case $m>1$. Therefore, we let $s_1(k) = 4k+5$ in this case.
\end{proof}

We are now in position to prove Theorems \ref{asymptotic stuff} and \ref{asymptotic stuff 2}.
By using the bounds on minor arcs from this section, we obtain an estimate for $\widetilde{G}_q(k)$.

\begin{proof}[Proof of Theorems \ref{asymptotic stuff} and \ref{asymptotic stuff 2}]
The result is an immediate consequence of combining our major arc estimates, Theorem \ref{thm major arcs}, and our minor arc estimates, Propositions ~\ref{prop s_1 first} and ~\ref{prop s_1 mpb}, from which we obtain $\widetilde{G}_q(k) \leq \max\{ s_1(k), 2k + 1 \}$.
We then simplify $s_1(k)$ from Propositions ~\ref{prop s_1 first} and ~\ref{prop s_1 mpb} via ~(\ref{defn of r}) to obtain the estimates given in the statement of Theorem \ref{asymptotic stuff}. When $k<p$, we see that
$s_1(k)$ given in ~(\ref{def s_1}) is identical to that defined for the integer case in \cite{W1}. Consequently, our estimates for
$\widetilde{G}_q(k)$ when $k < p$ are identical to the estimates of $\widetilde{G}(k)$ obtained in \cite{W1}.
\end{proof}

\section{Slim Exceptional Sets}
\label{exceptional section}

We carry out a similar calculation here as in Section \ref{asymptotic section} and obtain
Theorems \ref{exceptional stuff} and \ref{exceptional stuff 2}.
Recall from Section \ref{Introduction} that $\widetilde{E}_{s,k}(N, \psi)$ is defined to be
the set of $n \in I_N \cap \mathbb{J}_q^k[t]$ which satisfies
~(\ref{exceptional}).
As in \cite{W1}, 
we refer to a function $\psi(z)$
as being \textit{sedately increasing} when $\psi(z)$ is
a function of positive variable $z$ increasing monotonically to infinity,
and satisfying the condition that when $z$ is large, one has $\psi(z) = O(z^{\epsilon})$
for a positive number $\epsilon$ sufficiently small in the ambient context.
We also prove the following theorem on the estimate
of $|\widetilde{E}_{s,k}(N, \psi)|$ when $\psi$ is a sedately increasing function.
In order to avoid clutter in the exposition, we present the case $k = p^b+1$
separately from the rest of the cases.
\begin{thm}
\label{exceptional stuff 3}
Suppose $k \geq 3$ and $p \nmid k$. Suppose further that either $p \nmid (k-1)$ or $k = m p^b+1$, $m>1$. Let $\delta_0$ be as in the statement of Theorem \ref{Thm minor arc bound 1}.
If $\psi(z)$ is a sedately increasing function, then for $s \geq rk + r$ we have
$$
|\widetilde{E}_{s,k}(N, \psi)| \ll q^{(k - \delta_0 + \epsilon)P} \psi(q^P)^{2},
$$
where the implicit constant depends on
$s, q, k, \epsilon,  \mathcal{R}'$, and $\psi$.
\end{thm}

\begin{thm}
\label{exceptional stuff 3'}
Suppose $k \geq 3$ and $p \nmid k$. Suppose further that $k = p^b + 1$. Let
$$
\delta_0 = \frac{1}{16(p^b + 2)}.
$$
If $\psi(z)$ is a sedately increasing function, then for $s \geq 2k + 3$ we have
$$
|\widetilde{E}_{s,k}(N, \psi)| \ll q^{(k - \delta_0 + \epsilon)P} \psi(q^P)^{2},
$$
where the implicit constant depends on
$s, q, k, \epsilon, \mathcal{R}'$, and $\psi$.
\end{thm}

First, we consider the case when $p \nmid (k-1)$ in
Proposition \ref{prop u_2}. We then take care of the case $k = mp^b + 1$ in Proposition \ref{prop u_2 mpb}.

Let
$$
u'_0(j) = k^2 + 1 - \Bigg{\lceil} \frac{kj-2^{j-1}}{k+1-j} \Bigg{\rceil}.
$$
If $k < p$, we set
\begin{equation}
\label{def u_2}
u_2(k) = \min_{\stackrel{1 \leq j < k}{ 2^j \leq k(2k+1) }} u'_0(j).
\end{equation}
On the other hand, if  $k > p$ and $p \nmid (k-1)$, we set
\begin{equation}
\label{def u_2 second}
u_2(k) = rk  + 1 - \Bigg{\lceil}  \frac{ 3r-4}{k-2} \Bigg{\rceil}.
\end{equation}

\begin{prop}
\label{prop u_2}
Suppose $k \geq 3$, $p \nmid k$, 
and $p \nmid (k-1)$.
Let $u_2(k)$ be as given in ~(\ref{def u_2}) when $k < p$ and in ~(\ref{def u_2 second})
when  $k > p$. 
If $s \geq u_2(k)$, then there exists $\delta_2 > 0$  such that
$$
\int_{\mathfrak{m}} |g(\alpha)|^{2s} \ d \alpha
\ll q^{ ( 2s - k  -  \delta_2 ) X},
$$
where the implicit constant depends only on $s, q, k, \mathcal{R}'$, and $\delta_2$.
\end{prop}

\begin{proof}
Since the proof is similar to that of Proposition \ref{prop s_1 first}, we only give the set up of the proof here.
Let $h_0(k)$ be as in the statement of Lemma \ref{counting weyl}.
We let $2 u_0(j) = 2 r (k+1) a' + 2^j b'$, where $a'+b' = 1$. By H\"{o}lder's inequality,
we have
\begin{equation}
\label{ineq holder u}
\int_{\mathfrak{m}} |g(\alpha)|^{2 u_0(j)} \ d \alpha
\leq \left( \int_{\mathfrak{m}} |g(\alpha)|^{2r(k+1)} \ d \alpha  \right)^{a'}
\left( \oint |g(\alpha)|^{2^j} \ d \alpha \right)^{b'}.
\end{equation}

Recall that for the range of $k$ we are considering, we can take $\delta_0 = 1$ in Corollary \ref{minor arc bound 2}.
We consider $j$ in the following range: $1 \leq j < k$, $2^j \leq (2r-1)(k+1)+1$ and $j \leq h_0(k)$.
Define
$$
\eta(j) = \frac{rj}{k-j + 1} - \frac{2^{j-1} }{k-j + 1}
$$
and
let
$$
\gamma(j) = 1 + \eta(j) - \lceil \eta(j) \rceil.
$$

We choose
$$
a' = \frac{k-j}{k-j + 1} + \frac{ \gamma(j) } {r(k+1)-2^{j-1} }
$$
and
$$
b' = \frac{1}{k-j + 1} - \frac{ \gamma(j) } {r(k+1)-2^{j-1} }.
$$

Note that our restriction on $j$ ensures $b' > 0.$
Also, this choice of $a'$ and $b'$ ensures
$ a' - (k-j)b' > 0$. Then, by Corollary \ref{minor arc bound 2}
and ~(\ref{Hua}), we have the following bound for ~(\ref{ineq holder u}):
$$
\int_{\mathfrak{m}} |g(\alpha)|^{2 u_0(j)} \ d \alpha
\ll q^{\epsilon X} q^{a'(2r(k+1) - k - 1)X} q^{b'(2^j - j)X}
\ll q^{ ( 2 u_0(j) - k  -  ( a' - (k-j)b' ) + \epsilon ) X}.
$$
We then obtain the result by proceeding in a similar manner as in the proof of Proposition \ref{prop s_1 first}.
We leave verifying the remaining details of the proof as an exercise for the reader.
\end{proof}

Now we consider the case $k = m p^b + 1$. If $m=1$, we set
$u_2(k) = 2k+3$.
If $m>1$, then we set
\begin{equation}
\label{except eqn 2}
u_2(k) = rk + r -  \Bigg{\lfloor} \frac{(m-1)(1- 1/p)}{4} \Bigg{\rfloor}.
\end{equation}

\begin{prop}
\label{prop u_2 mpb}
Suppose $k = m p^b + 1$ with $p \nmid m$. Let $u_2(k)$ be $2k+3$
when $m=1$ and as in ~(\ref{except eqn 2}) when $m>1$.
If $s \geq u_2(k)$, then there exists $\delta_2 > 0$  such that
$$
\int_{\mathfrak{m}} |g(\alpha)|^{2s} \ d \alpha
\ll q^{ ( 2s - k  -  \delta_2 ) X},
$$
where the implicit constant depends only on $s, q, k, \mathcal{R}'$, and $\delta_2$.
\end{prop}

\begin{proof}
Since the proof is similar to that of Proposition \ref{prop s_1 mpb}, we only give the set up of the proof here.
For the case $m=1$, by a similar reasoning as in Proposition \ref{prop s_1 mpb}, we let
$u_2(k) = 2k+3$, and the result is an immediate consequence of Corollary \ref{minor arc bound 2'}.
We now deal with the case $m>1$. Let $h_0(k)$ be as in the statement of Lemma \ref{counting weyl}.
If $j \leq h_0(k)$, then for any $\epsilon > 0$ we have
~(\ref{Hua}).
We let $2 u_0(j) = 2 r (k+1) a' + 2^j b'$, where $a'+b' = 1$, as before in Proposition \ref{prop u_2}.
Then by H\"{o}lder's inequality, we have ~(\ref{ineq holder u}).
We consider $j$ in the following range: $1 \leq j < k$, $2^j \leq (2r-1)(k+1)+1$ and $j \leq h_0(k)$.


Let $\epsilon(j)$ be a small positive number.
We choose
$$
a' = \frac{k-j}{k-j + \delta} + \frac{ \epsilon(j) } {r(k+1)-2^{j-1} }
$$
and
$$
b' = \frac{\delta}{k-j + \delta} - \frac{ \epsilon(j) } {r(k+1)-2^{j-1} },
$$
where we let $\delta = \delta_0 = 1/(4p^b)$ from Corollary \ref{minor arc bound 2}.

Note that we pick $\epsilon(j)$ sufficiently small such that $b' > 0.$
Also, the range of $j$ we are considering and this choice of $a'$ and $b'$ ensure
$$
\delta a' - (k-j)b' = \frac{ (\delta + k -j) \epsilon(j) } {r(k+1)-2^{j-1} } > 0.
$$
By Corollary \ref{minor arc bound 2}
and ~(\ref{Hua}), we have the following bound for ~(\ref{ineq holder u}):
$$
\int_{\mathfrak{m}} |g(\alpha)|^{2 u_0(j)} \ d \alpha
\ll q^{\epsilon X} q^{a'(2r(k+1) - k - \delta)X} q^{b'(2^j - j)X}
\ll q^{ ( 2 u_0 (j) - k  -  ( \delta a' - (k-j)b' ) + \epsilon ) X}.
$$
We then obtain the result by proceeding in a similar manner as in the proof of Proposition \ref{prop s_1 mpb}.
We leave verifying the remaining details of the proof as an exercise for the reader.
\end{proof}

For $\psi(z)$ a function of positive variable $z$, recall we denote $\widetilde{E}_{s,k}(N, \psi)$ to be
the set of $n \in I_N \cap \mathbb{J}_q^k[t]$ for which
\begin{equation}
\label{ineq exceptional set}
\Big{|}  R_{s,k}(n) - \mathfrak{S}_{s,k}(n) J_{\infty}(n) q^{(s-k)P} \Big{|} > q^{(s-k)P} \psi(q^P)^{-1}.
\end{equation}

By Theorem \ref{thm major arcs}, for $s \geq 2k+1$ and any polynomial
$n \in \widetilde{E}_{s,k}(N, \psi)$ we have
\begin{equation}
\label{eqn majoor}
\int_{\mathfrak{M}} g(\alpha)^s e(-n \alpha) \ d\alpha = \mathfrak{S}_{s,k}(n) J_{\infty}(n) q^{(s-k)P}
+ \bigO{q^{(s-k- 2 \epsilon )P}},
\end{equation}
for sufficiently small $\epsilon > 0$.
Hence, it follows by ~(\ref{second integral}) that
\begin{eqnarray}
\label{major and minor - 1'}
R_{s,k}(n) 
&=& \mathfrak{S}_{s,k}(n) J_{\infty}(n) q^{(s-k)P}
+
\int_{\mathfrak{m}} g(\alpha)^s e(-n \alpha) \ d\alpha
\\
\notag
&&\phantom{123456789123456789123456789123} +
\bigO{q^{(s-k-2\epsilon)P}}.
\end{eqnarray}
By ~(\ref{ineq exceptional set}), ~(\ref{major and minor - 1'}) and the triangle
inequality, we see that there exists a constant $C_1 > 0$ such that
given any $n \in \widetilde{E}_{s,k}(N, \psi)$,
\begin{equation}
\label{ineq1}
\Big{|}  \int_{\mathfrak{m}} g(\alpha)^s e(-n \alpha) \ d\alpha  \Big{|} + C_1 q^{(s-k-2\epsilon)P}  > q^{(s-k)P} \psi(q^P)^{-1}.
\end{equation}
Suppose $\psi(z) < C_2 z^{\epsilon}$ for some constant $C_2 > 0$. Then it follows that $C_1 q^{(s-k-2\epsilon)P} < C_3 q^{(s-k-\epsilon)P} \psi(q^P)^{-1}$
for some constant $C_3>0$. Now there exists $M_0 > 0$ such that $C_3 q^{- \epsilon P} < 1/2$ for all $P \geq M_0$. Therefore, for $P$ sufficiently large we have that given any $n \in \widetilde{E}_{s,k}(N, \psi)$,
\begin{equation}
\label{ineq2}
\Big{|}  \int_{\mathfrak{m}} g(\alpha)^s e(-n \alpha) \ d\alpha  \Big{|} > \frac12 q^{(s-k)P} \psi(q^P)^{-1}.
\end{equation}


Let $E = | \widetilde{E}_{s,k}(N, \psi) |.$ Define the complex numbers $\eta(n)$, depending on $s$ and $k$,
for $n \in \widetilde{E}_{s,k}(N, \psi)$ by means of the equation
$$
\Big{|}  \int_{\mathfrak{m}} g(\alpha)^s e(-n \alpha) \ d\alpha  \Big{|}
= \eta(n) \int_{\mathfrak{m}} g(\alpha)^s e(-n \alpha) \ d\alpha.
$$
Clearly, $ | \eta(n) | = 1$ for all $n \in \widetilde{E}_{s,k}(N, \psi)$.
Define the exponential sum $K(\alpha)$ by
\begin{equation}
K(\alpha) = \sum_{  n \in \widetilde{E}_{s,k}(N, \psi)  }  \eta(n) e(n \alpha).
\end{equation}
Then, it follows from
~(\ref{ineq2}) that for $P$ sufficiently large
\begin{eqnarray}
\notag
\frac12 q^{(s-k)P} \psi(q^P)^{-1} E &<& \sum_{
n \in \widetilde{E}_{s,k}(N, \psi)
}
\eta(n) \int_{\mathfrak{m}} g(\alpha)^s e(-n \alpha) \ d\alpha
\\
\label{ineq3}
&=& \int_{\mathfrak{m}} g(\alpha)^s K(- \alpha) \ d\alpha.
\end{eqnarray}

We apply Cauchy-Schwartz inequality to the right hand side of ~(\ref{ineq3}) to obtain
\begin{equation}
\label{ineq 4}
\frac12 q^{(s-k)P} \psi(q^P)^{-1} E < \left( \int_{\mathfrak{m}} |g(\alpha)|^{2s} \ d\alpha \right)^{1/2} \left( \int_{\mathfrak{m}} |K( - \alpha)|^2 \ d\alpha \right)^{1/2}.
\end{equation}
We note that we have established the above inequality ~(\ref{ineq 4}) assuming $s \geq 2k+1$ here.
The orthogonality relation ~(\ref{orthogonality relation}) gives us
\begin{equation}
\label{on E}
\oint |K( \alpha)|^2 \ d\alpha =  \sum_{  n \in \widetilde{E}_{s,k}(N, \psi)  } 1 = E.
\end{equation}

With this set up, we are ready to prove Theorems \ref{exceptional stuff}, \ref{exceptional stuff 2}, \ref{exceptional stuff 3},
and \ref{exceptional stuff 3'}.
\begin{proof}[Proof of Theorems \ref{exceptional stuff}, \ref{exceptional stuff 2}, \ref{exceptional stuff 3}, and \ref{exceptional stuff 3'}]
Recall we defined $X = P+1$. By Propositions \ref{prop u_2} and \ref{prop u_2 mpb}, for $s \geq u_2(k)$ we know there exists
$\delta_2 > 0$ such that
$$
\left( \int_{\mathfrak{m}} |g(\alpha)|^{2s} \ d\alpha \right)^{1/2}
\ll
q^{(s - k/2 - \delta_2/2)P}.
$$
Therefore, we can further bound the right hand side of ~(\ref{ineq 4}) by the above inequality and ~(\ref{on E}),
and obtain for $s \geq \max \{ u_2(k), 2k+1 \}$,
$$
\frac12 q^{(s-k)P} \psi(q^P)^{-1} E^{1/2} < \left( \int_{\mathfrak{m}} |g(\alpha)|^{2s} \ d\alpha \right)^{1/2} \ll q^{(s - k/2 - \delta_2 / 2)P},
$$
which simplifies to
\begin{equation}
\label{bound on E first}
E \ll  q^{(k - \delta_2)P}  \psi(q^P)^{2}.
\end{equation}
Fix $\epsilon >0$ sufficiently small and let $\psi(z)$ be such that $\psi(q^P) \ll q^{\epsilon P/2}$. Then 
we have by ~(\ref{bound on E first}) that
$$
E \ll q^{(k - \delta_2 + \epsilon) P} <  q^{ \ord n -  (\delta_2  - \epsilon) P} \ll q^{N - (\delta_2  - \epsilon) \frac{N}{k} } = o(q^{N}).
$$
Therefore, we obtain $\widetilde{G}_q^+(k) \leq \max \{ u_2(k), 2k+1 \}$.
We then simplify $u_2(k)$ via ~(\ref{defn of r}) to obtain the estimates given in the statement of Theorem \ref{exceptional stuff}. When $k<p$, we see that
$u_2(k)$ given in ~(\ref{def u_2}) is identical to $u_1(k)$ defined in \cite{W1}. Consequently, our estimates for
$\widetilde{G}_q^+(k)$ when $k < p$ are identical to the estimates of $\widetilde{G}^+(k)$ obtained in \cite{W1}.
We have now completed the proof of Theorems \ref{exceptional stuff} and \ref{exceptional stuff 2}.

Finally, to prove Theorems \ref{exceptional stuff 3} and \ref{exceptional stuff 3'}, we substitute
~(\ref{on E}) into ~(\ref{ineq 4}), apply Corollary \ref{minor arc bound 2} or Corollary \ref{minor arc bound 2'} (depending on
$k$ and $p$), and obtain for $P$ sufficiently large
$$
\frac12 q^{(s-k)P} \psi(q^P)^{-1} E^{1/2} < \left( \int_{\mathfrak{m}} |g(\alpha)|^{2s} \ d\alpha \right)^{1/2} \ll q^{(s - k/2 - \delta_0 / 2 + \epsilon/2)P}.
$$
Rearranging the above inequality yields
$$
E \ll q^{(k - \delta_0 + \epsilon) P} \psi(q^P)^{2},
$$
as desired.
\end{proof}

\end{document}